\documentclass[a4paper,reqno]{amsart}

\newif\ifprivate
\privatetrue

\usepackage[utf8]{inputenc}
\usepackage{ae,aecompl}   
\usepackage[english]{babel}

\usepackage[wnafs,backref]{mathdef}

\ifprivate
\else
\renewcommand{\TODO}[1]{}
\fi

\numberwithin{equation}{section}
\numberwithin{figure}{section}
\numberwithin{table}{section}
\numberwithin{algorithm}{section}

\makeatletter \renewcommand\p@enumii{} \makeatother

\newcommand{\itemref}[1]{(\ref{#1})}

\newcommand{\wsubadditive}[1][w]{$#1$\nbd-subadditive}
\newcommand{\wsubadditivity}[1][w]{$#1$\nbd-subadditivity}
\newcommand{\wweaksubadditive}[1][w]{$#1$\nbd-weak-subadditive}

\theoremstyle{plain}

\begin{document}


\title[Optimality of the Width-$w$ Non-adjacent Form]{Optimality of the
  Width-$w$ Non-adjacent Form: General Characterisation and the Case of Imaginary Quadratic Bases}

\author{Clemens Heuberger}

\address{\parbox{12cm}{%
    Clemens Heuberger \\
    Institute of Optimisation and Discrete Mathematics (Math B) \\
    Graz University of Technology \\
    Steyrergasse 30/II, A-8010 Graz, Austria \\}} 

\email{\href{mailto:clemens.heuberger@tugraz.at}{clemens.heuberger@tugraz.at}}

\author{Daniel Krenn}

\thanks{The authors are supported by the Austrian Science Fund (FWF):
  W1230, Doctoral Program ``Discrete Mathematics''.}

\address{\parbox{12cm}{%
    Daniel Krenn \\
    Institute of Optimisation and Discrete Mathematics (Math B) \\
    Graz University of Technology \\
    Steyrergasse 30/II, A-8010 Graz, Austria \\}} 

\email{\href{mailto:mail@danielkrenn.at}{mail@danielkrenn.at} \textit{or}
  \href{mailto:krenn@math.tugraz.at}{krenn@math.tugraz.at}}




\begin{abstract}
  Efficient scalar multiplication in Abelian groups (which is an important
  operation in public key cryptography) can be performed using digital
  expansions. Apart from rational integer bases (double-and-add algorithm),
  imaginary quadratic integer bases are of interest for elliptic curve
  cryptography, because the Frobenius endomorphism fulfils a quadratic
  equation. One strategy for improving the efficiency is to increase the digit
  set (at the prize of additional precomputations). A common choice is the
  width\nbd-$w$ non-adjacent form (\wNAF): each block of $w$ consecutive
  digits contains at most one non-zero digit. Heuristically, this ensures a low
  weight, i.e.\ number of non-zero digits, which translates in few costly curve
  operations. This paper investigates the following question: Is the
  \wNAF{}-expansion optimal, where optimality means minimising the weight over
  all possible expansions with the same digit set?

  The main characterisation of optimality of \wNAF{}s can be formulated in the
  following more general setting: We consider an Abelian group together with an
  endomorphism (e.g., multiplication by a base element in a ring) and a finite
  digit set.  We show that each group element has an optimal \wNAF{}-expansion
  if and only if this is the case for each sum of two expansions of weight $1$.
  This leads both to an algorithmic criterion and to generic answers for
  various cases.

  Imaginary quadratic integers of trace at least $3$ (in absolute value) have
  optimal \wNAF{}s for $w\ge 4$. The same holds for the special case of base
  $(\pm 3\pm\sqrt{-3})/2$ and $w\ge 2$, which corresponds to Koblitz curves in
  characteristic three.  In the case of $\tau=\pm1\pm i$, optimality depends on
  the parity of $w$.  Computational results for small trace are given.

\end{abstract}


\maketitle


\section{Introduction}
\label{sec:intro}


Let $\tau$ be an imaginary quadratic algebraic integer. We consider \tauadic{}
(multi\nbd-)expansions for an element of $\Ztau$ using a redundant digit set,
i.e.\ our expansions need not be unique without any further constraints. The
question that arises is how to find ``good'' representations. This problem
comes from elliptic curve cryptography, where one is interested in expansions
leading to efficient calculation schemes.

A scalar multiplication, one of the key operations in in elliptic curve
cryptosystems, can be carried out by a double-and-algorithm or by using the
Frobenius endomorphism (``Frobenius-and-add''), cf.\
Koblitz~\cite{Koblitz:1992:cm} and
Solinas~\cite{Solinas:1997:improved-algorithm,Solinas:2000:effic-koblit}: The
Frobenius endomorphism $\varphi$ fulfils a quadratic equation
$\varphi^2-p\varphi+q=0$ for integers $p$ and $q$ depending on the curve. We
identify $\varphi$ with the complex root $\tau$ of the same equation. We
represent a scalar $n$ as $n=\sum_{j=0}^{\ell}\eta_j\tau^j$ for $\eta_j$ from
some suitable digit set $\cD$. Then the scalar multiplication $nP$ for some $P$
on the curve can be computed as $nP=\sum_{j=0}^{\ell}\eta_j\varphi^j(P)$. The
latter sum can be efficiently computed using Horner's scheme, where the $\eta
P$ for $\eta\in\cD$ have to be pre-computed.  The number of applications of
$\varphi$ (which is computationally cheap) corresponds to the length of the
expansion, the number of additions corresponds to the weight of the expansion,
i.e.\ the number of non-zero digits $\eta_j$.

Some of the results of this article do not depend on the setting in an
imaginary quadratic number field. Those are valid in the following more general
(abstract) setting and may be used in other situations as well.

A general number system is an Abelian group $\cA$ together with a group
endomorphism $\Phi$ and a digit set $\cD$, which is a finite subset of $\cA$
including $0$. The endomorphism acts as base in our number system. A common
choice is multiplication by a fixed element. In this general setting we
consider \emph{multi-expansions}, which are simply finite sums with summands
$\f{\Phi^k}{d}$, where $k\in\N_0$ and $d$ a non-zero digit in $\cD$. The
``multi'' in the expression ``multi-expansion'' means that we allow several
summands with the same~$k$. If the $k$ are pairwise distinct, we call the sum
an \emph{expansion}. Note that in the context of the Frobenius-and-add
algorithm, multi-expansions are as good as expansions, as long as the weight is
low.

A special expansion is the \emph{width\nbd-$w$ non-adjacent form}, or
\emph{\wNAF{}} for short. It will be the main concept throughout this
article. In a \wNAF{}-expansion, the $k$ in different summands differ at least
by a fixed positive rational integer~$w$, i.e.\ considering an expansion as a
sequence $\sequence{\eta_j}_{j\in\N_0} \in \cD^{\N_0}$ (or alternatively as
finite word over the alphabet $\cD$), each block $\eta_{j+w-1}\ldots\eta_{j}$
of width~$w$ contains at most one non-zero digit. The term ``non-adjacent
form'' at least goes back to Reitwiesner~\cite{Reitwiesner:1960}. More details
and precise definitions of those terms can be found in
Section~\ref{sec:exp-num-sys}.

Obviously, a \wNAF{} has low weight and therefore leads to quite efficient
scalar multiplication in the Frobenius-and-add method. The main question
investigated in this article is: Does the
\wNAF{} minimise the weight, i.e.\ the number of non-zero digits, among
all possible representations (multi-expansions) with the same digit set? If
the answer is affirmative, we call the \wNAF{}-expansion \emph{optimal}.

We give conditions equivalent to optimality in
Section~\ref{sec:abstract-optimality} in the setting of general number
systems. We show that each group element has an optimal \wNAF{}-expansion if
and only if the digit set is ``\wsubadditive{}'', which means that each
multi-expansion with two summands has a \wNAF{}-expansion with weight at
most~$2$. This condition can be verified algorithmically, since there are only
finitely many non-trivial cases to check. More precisely, one has to consider
the \wNAF{}s corresponding to $w (\card* \cD - 1)^2$ multi-expansions. Another
way to verify \wsubadditivity{} is to use the geometry of the digit
set. This is done in our imaginary quadratic setting, see below.

Now consider some special cases of number systems, where optimality or
non-optimality of the non-adjacent form is already known. Here, multiplication
by a base element is chosen as endomorphism $\Phi$. In the case of \wNAF[2]{}s
with digit set $\set{-1,0,1}$ and base~$2$, optimality is known, cf.\
Reitwiesner~\cite{Reitwiesner:1960}. This was reproved in Jedwab and
Mitchell~\cite{Jedwab-Mitchell:1989} and in Gordon~\cite{Gordon:1998}. That
result was generalised in Avanzi~\cite{avanzi:mywnaf}, Muir and
Stinson~\cite{muirstinson:minimality} and in Phillips and
Burgess~\cite{Phillips-Burgess:2004:minim-weigh}. There, the optimality of the
\wNAF{}s with base~$2$ was shown. As digit set, zero and all odd numbers with
absolute value less than $2^{w-1}$ were used. In this setting, there is also
another optimal expansion, cf.\ Muir and
Stinson~\cite{Muir-Stinson:2005:new-minim}. Using base~$2$ and a digit set
$\set{0,1,x}$ with $x\in\Z$, optimality of the \wNAF[2]{}s is answered in
Heuberger and Prodinger~\cite{Heuberger-Prodinger:2006:analy-alter}. Some of
these results will be reproved and extended to arbitrary rational integer bases
with our tools in Section~\ref{sec:opt-base-2}. That proof will show the main
idea how to use the geometry of the digit set to show \wsubadditivity{} and
therefore optimality.

We come back to our imaginary quadratic setting, so suppose that the imaginary
quadratic base~$\tau$ is a solution of $\tau^2 - p \tau + q = 0$, where $p$ and
$q$ are rational integers with $q>p^2/4$. Here, $\Ztau$ plays the r\^ole of the
group and multiplication by $\tau$ is taken as the endomorphism. We
suppose that the digit set consists of $0$ and one representative of minimal
norm of every residue class modulo $\tau^w$, which is not divisible by $\tau$,
and we call it a \emph{minimal norm representatives digit set}. It can be shown
that, taking such a digit set, every element of $\Ztau$ admits a unique
\wNAF{}, cf.\ Blake, Kumar Murty and Xu~\cite{Blake-Murty-Xu:ta:nonad-radix,
  Blake-Murty-Xu:2005:naf, Blake-Kumar-Xu:2005:effic-algor},
Koblitz~\cite{Koblitz:1998:ellip-curve} and
Solinas~\cite{Solinas:1997:improved-algorithm, Solinas:2000:effic-koblit} for
some special cases or Heuberger and
Krenn~\cite{Heuberger-Krenn:2010:wnaf-analysis} for a general result. All those
definitions and basics can be found in Section~\ref{sec:digit-sets-iq-bases} in
a precise formulation.

First, consider the cases $\abs{p}=1$ and $q=2$, which comes from a Koblitz
curve in characteristic~$2$, cf.\ Koblitz~\cite{Koblitz:1992:cm}, Meier and
Staffelbach~\cite{Meier-Staffelbach:1993:effic}, and
Solinas~\cite{Solinas:1997:improved-algorithm,
  Solinas:2000:effic-koblit}. There optimality of the \wNAF{}s can be shown for
$w\in\set{2,3}$, cf.\ Avanzi, Heuberger and
Prodinger~\cite{Avanzi-Heuberger-Prodinger:2006:minim-hammin,
  Avanzi-Heuberger-Prodinger:2006:scalar-multip-koblit-curves}. The case $w=2$
can also be found in Gordon~\cite{Gordon:1998}. For the cases
$w\in\set{4,5,6}$, non-optimality was shown, see
Heuberger~\cite{Heuberger:2010:nonoptimality}.

In the present paper we give a general result on the optimality of the \wNAF{}s
with imaginary quadratic bases, namely when $\abs{p}\geq3$, as well as some
results for special cases. So let $\abs{p}\geq3$. If $w\geq4$, then optimality
of the \wNAF{}s could be shown in all cases. If we restrict to $\abs{p}\geq5$,
then the \wNAF{}s are already optimal for $w\geq3$. Further, we give a
condition --- $p$ and $q$ have to fulfil a special inequality --- given, when
\wNAF[2]{}s are optimal. All those results can be found in
Section~\ref{sec:opt:i-q}. There we show that the digit set in that cases is
\wsubadditive{} by using its geometry.

In the last four sections some special cases are examined. Important ones are
the cases $\abs{p}=3$ and $q=3$ coming from Koblitz curves in
characteristic~$3$.  In Kröll~\cite{Kroell:ta:optim-of} optimality of the
\wNAF{}s was shown for $w\in\set{2,3,4,5,6,7}$ by using a transducer and some
heavy symbolic computations. In this article we prove that the
\wNAF{}-expansions are optimal for all $w\geq2$, see
Section~\ref{sec:case-p-3-q-3}. In Section~\ref{sec:case-p-2-q-2} we look at
the cases $\abs{p}=2$ and $q=2$. There the \wNAF{}-expansions are optimal if
and only if $w$ is odd. In the cases $p=0$ and $q\geq2$, see
Section~\ref{sec:case-p-0}, non-optimality of the \wNAF{}s with odd $w$ could
be shown.


\section{Expansions and Number Systems}
\label{sec:exp-num-sys}


This section contains the abstract definition of number systems and the
definition of expansions. Further, we specify the width-$w$ non-adjacent
form and notions related to it.

Abstract number systems can be found in van de
Woestijne~\cite{Woestijne:2009:struc-of}, which are generalisations of the
number systems used, for example, in Germ\'an and
Kov\'acs~\cite{German-Kovacs:2007:number-system-const}. We use that concept to
define \wNAF{}-number systems.


\begin{definition}
  A \emph{pre-number system} is a triple $(\cA,\Phi,\cD)$ where $\cA$ is an
  Abelian group, $\Phi$ an endomorphism of $\cA$ and the \emph{digit set}
  $\cD$ is a subset of $\cA$ such that $0\in\cD$ and each non-zero digit is not
  in the image of $\Phi$.
\end{definition}


Note that we can assume $\Phi$ is not surjective, because otherwise the digit
set would only consist of $0$.


Before we define expansions and multi-expansions, we give a short introduction
on multisets. We take the notation used, for example, in
Knuth~\cite{Knuth:1998:Art:2}.


\begin{notation}
  A \emph{multiset} is like a set, but identical elements are allowed to appear
  more than once. For a multiset $A$, its cardinality $\card*{A}$ is the number
  of elements in the multiset. For multisets $A$ and $B$, we define new
  multisets $A \uplus B$ and $A \setminus B$ in the following way: If an element
  occurs exactly $a$ times in $A$ and $b$ times in $B$, then it occurs exactly
  $a+b$ times in $A \uplus B$ and it occurs exactly $\f{\max}{a-b,0}$ times in
  $A \setminus B$.
\end{notation}


Now a pre-number system (and multisets) can be used to define what
expansions and multi-expansions are.


\begin{definition}[Expansion]\label{def:expansion}
  Let $(\cA,\Phi,\cD)$ be a pre-number system, and let $\bfmu$ be a multiset
  with elements $(d,n) \in (\cD\setminus\set{0}) \times \N_0$. We define the
  following:
  \begin{enumerate}
  \item We set
    \begin{equation*}
      \expweight{\bfmu} := \card*{\bfmu}
    \end{equation*}
    and call it the \emph{Hamming-weight of $\bfmu$} or simply \emph{weight of
      $\bfmu$}. The multiset $\bfeta$ is called \emph{finite}, if its weight is
    finite.

  \item We call an element $(d,n)\in\bfmu$ a \emph{singleton} and
    $\f{\Phi^n}{d}$ the \emph{value of the singleton $(d,n)$}.

  \item Let $\bfmu$ be finite. We call
    \begin{equation*}
      \expvalue{\bfmu} := \sum_{(d,n)\in\bfmu} \f{\Phi^n}{d}
    \end{equation*}
    the \emph{value of $\bfmu$}.

  \item Let $z\in\cA$. A \emph{multi-expansion of $z$} is a finite $\bfmu$
    with $\expvalue{\bfmu} = z$.

  \item Let $z\in\cA$. An \emph{expansion of $z$} is a multi-expansion $\bfmu$
    of $z$ where all the $n$ in $(d,n)\in\bfmu$ are pairwise distinct.
  \end{enumerate}
\end{definition}


We use the following conventions and notations. If necessary, we see a
singleton as a multi-expansion or an expansion of weight~$1$. We identify an
expansion $\bfeta$ with the sequence $\sequence{\eta_n}_{n\in\N_0} \in
\cD^{\N_0}$, where $\eta_n = d$ for $(d,n) \in \bfeta$ and all other
$\eta_n=0$. For an expansion $\bfeta$ (usually a bold, lower case Greek letter)
we will use $\eta_n$ (the same letter, but indexed and not bold) for the
elements of the sequence. Further, we identify expansions (sequences) in
$\cD^{\N_0}$ with finite words over the alphabet $\cD$ written from right
(least significant digit) to left (most significant digit),
except left-trailing zeros, which are usually skipped.


Note, if $\bfeta$ is an expansion, then the weight of $\bfeta$ is
\begin{equation*}
  \expweight{\bfeta} = \card*{\set*{n\in\N_0}{\eta_n\neq0}}
\end{equation*}
and the value of $\bfeta$ is
\begin{equation*}
  \expvalue{\bfeta} = \sum_{n\in\N_0} \f{\Phi^n}{\eta_n}.
\end{equation*}


For the sake of completeness --- although we do not need it in this paper --- a
pre-number system is called \emph{number system} if each element of $\cA$ has
an expansion. We call the number system \emph{non-redundant} if there is
exactly one expansion for each element of $\cA$, otherwise we call it
\emph{redundant}. We will modify this definition later for \wNAF{} number
systems.


Before going any further, we want to see some simple examples for the given
abstract definition of a number system. We use multiplication by an
element~$\tau$ as endomorphism~$\Phi$. This leads to values of the type
\begin{equation*}
  \expvalue{\bfeta} = \sum_{n\in\N_0} \eta_n\tau^n
\end{equation*}
for an expansion $\bfeta$.


\begin{example}
  The binary number system is the pre-number system
  \begin{equation*}
    (\N_0,z\mapsto 2z,\set{0,1}).
  \end{equation*}
  It is a non-redundant number system, since each integer admits exactly one
  binary expansion. We can extend the binary number system to the pre-number
  system
  \begin{equation*}
    (\Z,z\mapsto 2z,\set{-1,0,1}),
  \end{equation*}
  which is a redundant number system.
\end{example}


In order to get a non-redundant number system out of a redundant one, one can
restrict the language, i.e.\ we forbid some special configurations in an
expansion. There is one special kind of expansion, namely the non-adjacent
form, where no adjacent non-zeros are allowed. A generalisation of it is
defined here.


\begin{definition}[Width-$w$ Non-Adjacent Form]\label{def:wnaf}
  Let $w$ be a positive integer and $\cD$ be a digit set (coming from a
  pre-number system). Let $\bfeta = \sequence{\eta_j}_{j\in\N_0} \in
  \cD^{\N_0}$. The sequence $\bfeta$ is called a \emph{width\nbd-$w$
    non-adjacent form}, or \emph{\wNAF{}} for short, if each factor
  $\eta_{j+w-1}\ldots\eta_j$, i.e.\ each block of length $w$, contains at most
  one non-zero digit.

  A \wNAF{}-expansion is an expansion that is also a \wNAF{}.
\end{definition}


Note that a \wNAF{}-expansion is finite. With the previous definition we can now
define what a \wNAF{} number system is.


\begin{definition}
  Let $w$ be a positive integer. A pre-number system $(\cA,\Phi,\cD)$ is called
  a \emph{\wNAF{} number system} if each element of $\cA$ admits a
  \wNAF{}-expansion, i.e.\ for each $z\in\cA$ there is a \wNAF{} $\bfeta \in
  \cD^{\N_0}$ with $\expvalue{\bfeta} = z$. We call a \wNAF{} number system
  \emph{non-redundant} if each element of $\cA$ has a unique \wNAF{}-expansion,
  otherwise we call it \emph{redundant}.
\end{definition}


Now we continue the example started above.


\begin{example}\label{ex:wNAF-base-2}

  The redundant number system
  \begin{equation*}
    (\Z,z\mapsto 2z,\set{-1,0,1})
  \end{equation*}
  is a non-redundant \wNAF[2]{} number system. This fact has been shown in
  Reitwiesner~\cite{Reitwiesner:1960}. More generally, for an integer~$w$ at
  least~$2$, the number system
  \begin{equation*}
    (\Z,z\mapsto 2z,\cD), 
  \end{equation*}
  where the digit set $\cD$ consists of $0$ and all odd integers with absolute
  value smaller than $2^{w-1}$, is a non-redundant \wNAF{} number system, cf.\
  Solinas~\cite{Solinas:1997:improved-algorithm, Solinas:2000:effic-koblit} or
  Muir and Stinson~\cite{muirstinson:minimality}.
\end{example}


Finally, since this paper deals with the optimality of expansions, we have to
define the term ``optimal''. This is done in the following definition.


\begin{definition}[Optimal Expansion]\label{def:optimal}
  Let $(\cA,\Phi,\cD)$ be a pre-number system, and let $z\in\cA$. A
  multi-expansion or an expansion $\bfmu$ of $z$ is called \emph{optimal} if
  for any multi-expansion $\bfnu$ of $z$ we have
  \begin{equation*}
    \expweight{\bfmu} \leq \expweight{\bfnu},
  \end{equation*}
  i.e.\ $\bfmu$ minimises the Hamming-weight among all multi-expansions of
  $z$. Otherwise $\bfmu$ is called \emph{non-optimal}.
\end{definition}

The ``usual'' definition of optimal, cf.\ \cite{Reitwiesner:1960,
  Jedwab-Mitchell:1989, Gordon:1998, avanzi:mywnaf, muirstinson:minimality,
  Phillips-Burgess:2004:minim-weigh, Muir-Stinson:2005:new-minim,
  Heuberger-Prodinger:2006:analy-alter,
  Avanzi-Heuberger-Prodinger:2006:minim-hammin,
  Avanzi-Heuberger-Prodinger:2006:scalar-multip-koblit-curves,
  Heuberger:2010:nonoptimality}, is more restrictive: An expansion of $z\in\cA$
is optimal if it minimises the weight among all expansions of $z$. The
difference is that in Definition~\ref{def:optimal} we minimise over all
multi-expansions. Using (multi-)expansions come from an application: we want to
do efficient operations. There it is no problem to take multi-expansions if
they are ``better'', so it is more natural to minimise over all of them instead
of just over all expansions.


\section{The Optimality Result}
\label{sec:abstract-optimality}


This section contains our main theorem, the Optimality Theorem,
Theorem~\ref{thm:optimality-thm}. It contains four equivalences. One of it is a
condition on the digit set, one is optimality of the \wNAF{}. We start with the
definition of that condition on the digit set.


\begin{definition}
  Let $(\cA,\Phi,\cD)$ be a pre-number system, and let $w$ be a positive
  integer. We say that the digit set $\cD$ is \emph{\wsubadditive{}} if the sum
  of the values of two singletons has a \wNAF{}-expansion of weight at
  most~$2$.
\end{definition}


In order to verify the \wsubadditivity{}-condition it is enough to check
singletons $(c,0)$ and $(d,n)$ with $n\in\set{0,\dots,w-1}$ and non-zero
digits~$c$ and~$d$. Therefore, one has to consider $w \left( \card* \cD - 1
\right)^2$ multi-expansions.


\begin{theorem}[Optimality Theorem]\label{thm:optimality-thm}
  Let $(\cA,\Phi,\cD)$ be a pre-number system with
  \begin{equation*}
    \bigcap_{m\in\N_0} \f{\Phi^m}{\cA} = \set{0},
  \end{equation*}
  and let $w$ be a positive integer. Then the following statements are
  equivalent:
  \begin{enumequivalences}
  \item\label{enu:thm-add:subadditive} The digit set $\cD$ is \wsubadditive{}.
  \item\label{enu:thm-add:singletons} For all multi-expansions $\bfmu$ there is
    a \wNAF{}-expansion $\bfxi$ such that
    \begin{equation*}
      \expvalue{\bfxi} = \expvalue{\bfmu}
    \end{equation*}
    and
    \begin{equation*}
      \expweight{\bfxi} \leq \expweight{\bfmu}.
    \end{equation*}
  \item\label{enu:thm-add:wNAFs} For all \wNAF{}-expansions $\bfeta$ and
    $\bfvartheta$ there is a \wNAF{}-expansion $\bfxi$ such that
    \begin{equation*}
      \expvalue{\bfxi} = \expvalue{\bfeta} + \expvalue{\bfvartheta}
    \end{equation*}
    and
    \begin{equation*}
      \expweight{\bfxi} \leq \expweight{\bfeta} + \expweight{\bfvartheta}.
    \end{equation*}
  \item\label{enu:thm-add:optimal} If $z\in\cA$ admits a multi-expansion, then
    $z$ also admits an optimal \wNAF{}-expansion.
  \end{enumequivalences}
\end{theorem}


Note that if we assume that each element $\cA$ has at least one expansion (e.g.\
by assuming that we have a \wNAF{} number system), then we have the equivalence
of \wsubadditivity{} of the digit set and the existence of an optimal
\wNAF{}-expansion for each group element.


We will use the term ``addition'' in the following way: The addition of two
group elements $x$ and $y$ means finding a \wNAF{}-expansion of the sum
$x+y$. Addition of two multi-expansions shall mean addition of their values.


\begin{proof}[Proof of Theorem~\ref{thm:optimality-thm}]
  For a non-zero $z\in\cA$, we define
  \begin{equation*}
    \f{L}{z} := \max\set*{m\in \N_0}{z\in \f{\Phi^m}{\cA}}.
  \end{equation*}
  The function $L$ is well-defined, because
  \begin{equation*}
    \bigcap_{m\in\N_0} \f{\Phi^m}{\cA} 
    = \set{0}.
  \end{equation*}  
  We show that \textit{(\ref{enu:thm-add:subadditive})} implies
  \textit{(\ref{enu:thm-add:singletons})} by induction on the pair
  $(\expweight{\bfmu}, \f{L}{\expvalue{\bfmu}})$ for the
  multi-expansion~$\bfmu$. The order on that pairs is lexicographically. In the
  case $\expvalue{\bfmu}=0$, we choose $\bfxi=0$ and are finished. Further, if
  the multi-expansion $\bfmu$ consists of less than two elements, then there is
  nothing to do, so we suppose $\expweight{\bfeta} \geq 2$.

  We choose a singleton $(d,n)\in\bfmu$ (note that $d\in\cD\setminus\set{0}$
  and $n\in\N_0$) with minimal $n$ and set $\bfmu^\bfstar := \bfmu \setminus
  \set{(d,n)}$. If $n>0$, then we consider the multi-expansion $\bfmu'$ arising
  from $\bfmu$ by shifting all indices by $n$, use the induction hypothesis on
  $\bfmu'$ and apply $\Phi^n$. Note that $\bfmu'$ and $\bfmu$ have the same
  weight, but
  \begin{equation*}
    \f{L}{\expvalue{\bfmu'}}=\f{L}{\expvalue{\bfmu}}-n<\f{L}{\expvalue{\bfmu}}.
  \end{equation*}
  So we can assume $n=0$. Using the induction hypothesis, there is a
  \wNAF{}-expansion $\bfeta$ of $\expvalue{\bfmu^\bfstar}$ with weight strictly
  smaller than $\expweight{\bfmu}$.

  Consider the addition of $\bfeta$ and the digit $d$. If the digits
  $\eta_\ell$ are zero for all $\ell\in\set{0,\dots,w-1}$ , then the result
  follows by setting $\bfxi = \ldots\eta_{w+1}\eta_w0^{w-1}d$. So we can assume
  \begin{equation*}
    \bfeta = \bfbeta 0^{w-k-1} b 0^k
  \end{equation*}
  with a \wNAF{} $\bfbeta$, a digit $b\neq0$ and $k\in\set{0,\dots,w-1}$.
  Since the digit set $\cD$ is \wsubadditive{}, there is a \wNAF{} $\bfgamma$ of
  $\f{\Phi^k}{b}+d$ with weight at most~$2$. If the weight is strictly smaller
  than~$2$, we use the induction hypothesis on the multi-expansion $\bfbeta 0^w
  \uplus \bfgamma$ to get a \wNAF{} $\bfxi$ with the desired properties and are
  done. Otherwise, denoting by $J$ the smallest index with $\gamma_J\neq0$,
  we distinguish between two cases: $J = 0$ and $J>0$.

  First let $J=0$. The \wNAF{} $\bfbeta$ (seen as multi-expansion) has a weight
  less than $\bfeta$, so, by induction hypothesis, there is a \wNAF{} $\bfxi'$
  with
  \begin{equation*}
    \expvalue{\bfxi'} = 
    \expvalue{\bfbeta} + \expvalue{\ldots\gamma_{w+1}\gamma_w}
  \end{equation*}
  and
  \begin{equation*}
    \expweight{\bfxi'} 
    \leq \expweight{\bfbeta} + \expweight{\ldots\gamma_{w+1}\gamma_w}.
  \end{equation*}
  We set $\bfxi = \bfxi'\gamma_{w-1}\ldots\gamma_0$. Since $\bfxi$ is a
  \wNAF{}-expansion we are finished, because
  \begin{multline*}
    \expvalue{\bfxi} =
    \f{\Phi^w}{\expvalue{\bfbeta}} 
    + \f{\Phi^w}{\expvalue{\ldots\gamma_{w+1}\gamma_w}}
    + \expvalue{\gamma_{w-1}\ldots\gamma_0} \\
    = \f{\Phi^w}{\expvalue{\bfbeta}} + \f{\Phi^k}{b} + d
    = \expvalue{\bfeta} + d
    = \expvalue{\bfmu^\bfstar} + d
    = \expvalue{\bfmu}
  \end{multline*}
  and
  \begin{multline*}
    \expweight{\bfxi}
    = \expweight{\bfxi'} + \expweight{\gamma_{w-1}\ldots\gamma_0}
    \leq \expweight{\bfbeta} + \expweight{\bfgamma} \\
    \leq \expweight{\bfeta} + 1 
    \leq \expweight{\bfmu^\bfstar} + 1 
    = \expweight{\bfmu}.
  \end{multline*}

  Now, in the case $J>0$, we consider the multi-expansion $\bfnu :=\bfbeta 0^w
  \uplus \bfgamma$. We use the induction hypothesis for $\bfnu$ shifted by $J$
  (same weight, $L$ decreased by $J$) and apply $\Phi^J$ on the result.

  The proofs of the other implications of the four equivalences are simple. To
  show that \textit{(\ref{enu:thm-add:singletons})} implies
  \textit{(\ref{enu:thm-add:wNAFs})}, take $\bfmu := \bfeta \uplus \bfvartheta$,
  and \textit{(\ref{enu:thm-add:wNAFs})} implies
  \textit{(\ref{enu:thm-add:subadditive})} is the special case when $\bfeta$ and
  $\bfvartheta$ are singletons.

  Further, for \textit{(\ref{enu:thm-add:singletons})} implies
  \textit{(\ref{enu:thm-add:optimal})} take an optimal multi-expansion $\bfmu$
  (which exists, since $z$ admits at least one multi-expansion). We get a
  \wNAF{}-expansion $\bfxi$ with $\expweight{\bfxi} \leq
  \expweight{\bfmu}$. Since $\bfmu$ was optimal, equality is obtained in the
  previous inequality, and therefore $\bfxi$ is optimal, too. The
  converse, \textit{(\ref{enu:thm-add:optimal})} implies
  \textit{(\ref{enu:thm-add:singletons})}, follows using $z=\expvalue{\bfmu}$
  and the property that optimal expansions minimise the weight.
\end{proof}


\begin{proposition}\label{pro:suff-cond-wsubadditive}
  Let $(\cA,\Phi,\cD)$ be a pre-number system with
 \begin{equation*}
    \bigcap_{m\in\N_0} \f{\Phi^m}{\cA} = \set{0},
  \end{equation*}
  and let $w$ be a positive integer. We have the following sufficient
  condition: Suppose we have sets $U$ and $S$ such that $\cD \subseteq U$,
  $-\cD \subseteq U$, $U \subseteq \f{\Phi}{U}$ and all elements in $S$ are
  singletons. If $\cD$ contains a representative for each residue class modulo
  $\f{\Phi^w}{\cA}$ which is not contained in $\f{\Phi}{\cA}$ and
  \begin{equation}\label{eq:suff-cond-wsubadditive}
    \left( \f{\Phi^{w-1}}{U} + U + U \right)
    \cap \f{\Phi^w}{\cA} \subseteq S \cup \set{0},
  \end{equation}
  then the digit set $\cD$ is \wsubadditive{}.
\end{proposition}


Sometimes it is more convenient to use~\eqref{eq:suff-cond-wsubadditive} of
this proposition instead of the definition of \wsubadditive{}. For example, in
Section~\ref{sec:opt-base-2} all digits lie in an interval $U$ and
all integers in that interval $S=U$ have a \wNAF{} expansion with weight
at most~$1$. The same technique is used in the optimality result of
Section~\ref{sec:opt:i-q}.


\begin{proof}[Proof of Proposition~\ref{pro:suff-cond-wsubadditive}]
  Let $(c,0)$ and $(d,n)$ be singletons with $n\in\set{0,\dots,w-1}$ and
  consider $y = \expvalue{(c,0) \uplus (d,n)}$. If $y=0$, we have nothing to
  do, so we can assume $y\neq0$. First suppose $y\not\in\f{\Phi}{\cA}$. Because
  of our assumptions on $\cD$ there is a digit $a$ such that
  \begin{equation*}
    z := \f{\Phi^n}{d} + c - a \in \f{\Phi^w}{\cA}.
  \end{equation*}
  If $z$ is not zero, then, using our sufficient condition, there is a
  singleton $(b,m)$ with value~$z$, and we have $m \geq w$. The
  \wNAF{}-expansion $b0^{m-1}a$ does what we want.

  Now suppose $y\in\f{\Phi^k}{\cA}$ with a positive integer $k$, which is
  chosen maximally. That case can only happen when $n=0$. Since $y\neq0$ and
  our assumptions on $\cD$ there is a \wNAF{}-expansion of~$y$ with a singleton
  $(a,k)$ as least significant digit. If $k\in\set{0,\dots,w-1}$, then
  \begin{equation*}
    z := d + c -  \f{\Phi^k}{a} \in \f{\Phi^{w+k}}{\cA}.
  \end{equation*}
  Then a non-zero~$z$ is the value of a singleton $(b,m)$, $m \geq w+k$,
  because of~\eqref{eq:suff-cond-wsubadditive}, and we obtain a
  \wNAF{}-expansion of~$y$ with singletons $(b,m)$ and $(a,k)$. If $k \geq w$,
  then
  \begin{equation*}
    z := d + c \in \f{\Phi^{w+k}}{\cA}
  \end{equation*}
  and $z$ is the value of a singleton $(b,m)$
  by~\eqref{eq:suff-cond-wsubadditive}. We get a \wNAF{}-expansion~$b0^m$.
\end{proof}


Sometimes the \wsubadditivity{}-condition is a bit too strong, so we do net get
optimal \wNAF{}s. In that case one can check whether \wNAF[(w-1)]{}s are
optimal. This is stated in the following remark, where the
\wsubadditive{}-condition is weakened.


\begin{remark}\label{rem:wweaksubadditive}
  Suppose that we have the same setting as in
  Theorem~\ref{thm:optimality-thm}. We call the digit set
  \emph{\wweaksubadditive{}} if the sum of the values of two singletons $(c,m)$
  and $(d,n)$ with $\abs{m-n} \neq w-1$ has a \wNAF{}-expansion with weight at
  most~$2$.
  
  We get the following result: If the digit set $\cD$ is \wweaksubadditive{},
  then each element of $\cA$, which has at least one multi-expansion, has an
  optimal \wNAF[(w-1)]{}-expansion. The proof is similar to the proof of
  Theorem~\ref{thm:optimality-thm}, except that a ``rewriting'' only happens
  when we have a \wNAF[(w-1)]{}-violation.
\end{remark}


\section{Optimality for Integer Bases}
\label{sec:opt-base-2}


In this section we give a first application of the abstract optimality theorem
of the previous section. We reprove the optimality of the \wNAF{}s with a
minimal norm digit set and base~$2$. But the result is more general: We prove
optimality for all integer bases (with absolute value at least~$2$). This
demonstrates one basic idea how to check whether a digit set is \wsubadditive{}
or not.


Let $b$ be an integer with $\abs{b}\geq2$ and $w$ be an integer with $w\geq2$.
Consider the non-redundant \wNAF{} number system
\begin{equation*}
  (\Z,z\mapsto bz,\cD)
\end{equation*}
where the digit set $\cD$ consists of $0$ and all integers with absolute value
strictly smaller than $\frac12 \abs{b}^w$ and not divisible by~$b$. We
mentioned the special case base~$2$ of that number system in
Example~\ref{ex:wNAF-base-2}. See also Reitwiesner~\cite{Reitwiesner:1960} and
Solinas~\cite{Solinas:2000:effic-koblit}.


The following optimality result can be shown. For proofs of the base~$2$
setting cf.\ Reitwiesner~\cite{Reitwiesner:1960}, Jedwab and
Mitchell~\cite{Jedwab-Mitchell:1989}, Gordon~\cite{Gordon:1998},
Avanzi~\cite{avanzi:mywnaf}, Muir and Stinson~\cite{muirstinson:minimality},
and Phillips and Burgess~\cite{Phillips-Burgess:2004:minim-weigh}.


\begin{theorem}
  With the setting above, the \wNAF{}-expansion for each integer is optimal.
\end{theorem}


\begin{proof}
  We show that the digit set $\cD$ is \wsubadditive{} by verifying the
  sufficient condition of Proposition~\ref{pro:suff-cond-wsubadditive}. Then
  optimality follows from Theorem~\ref{thm:optimality-thm}. First, note that
  the \wNAF{}-expansion of each integer with absolute value at most
  $\abs{b}^{w-1}$ has weight at most~$1$, because either the integer is already
  a digit, or one can divide by a power of~$b$ to get a digit. Further, we have
  $\cD = -\cD$. Therefore it suffices to show that
  \begin{equation*}
    b^{-w} \left(b^{w-1}\cD + \cD + \cD\right) \subseteq 
    \intervalcc{-\tfrac12\abs{b}^w}{\tfrac12 \abs{b}^w}.
  \end{equation*}

  Let
  \begin{equation*}
    b^wz = b^{w-1} c + a + d
  \end{equation*}
  for some digits $a$, $c$ and $d$. A digit has absolute value less
  than $\frac12 \abs{b}^w$, so
  \begin{equation*}
     \abs{z} 
     < \abs{b}^{-w} \left(\abs{b}^k + 2\right) \tfrac12 \abs{b}^w
     \leq \abs{b}^{w-1},
  \end{equation*}
  which was to show.
\end{proof}


\section{Voronoi Cells}
\label{sec:voronoi}


We first start to define Voronoi cells. Let $\tau\in\C$ be
an algebraic integer, imaginary quadratic, i.e.\ $\tau$ is solution of an
equation $\tau^2 - p \tau + q = 0$ with $p,q\in\Z$ and such that
$q-p^2/4>0$. 


\begin{definition}[Voronoi Cell]\label{def:voronoi}
  We set
  \begin{equation*}
    V := \set*{z\in\C}{ \forall y\in\Ztau \colon \abs{z} \leq \abs{z-y}}
  \end{equation*}
  and call it the \emph{Voronoi cell for $0$} corresponding to the set
  $\Ztau$. Let $u \in \Ztau$. We define the \emph{Voronoi cell for $u$} as
  \begin{equation*}
    V_u := u + V = \set*{u+z}{z \in V}
    = \set*{z\in\C}{ \forall y\in\Ztau \colon \abs{z-u} \leq \abs{z-y}}.  
  \end{equation*}
  The point $u$ is called \emph{centre of the Voronoi cell} or \emph{lattice
    point corresponding to the Voronoi cell}.
\end{definition}


\begin{figure}
  \centering
  \includegraphics{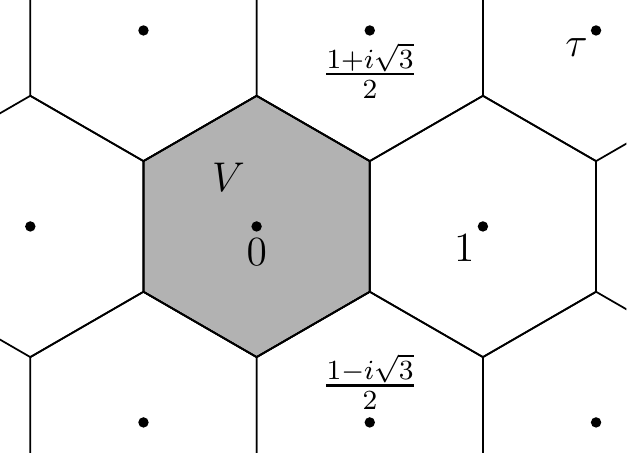}
  \caption[Voronoi cell $V$ for $0$]{Voronoi cell $V$ for $0$ corresponding to
    the set $\Ztau$ with $\tau = \frac{3}{2}+\frac{i}{2}\sqrt{3}$.}
  \label{fig:voronoi}
\end{figure}


An example of a Voronoi cell in a lattice $\Ztau$ is shown in
Figure~\ref{fig:voronoi}. Two neighbouring Voronoi cells have at most a
subset of their boundary in common. This can be a problem, when we tile the
plane with Voronoi cells and want that each point is in exactly one cell. To
fix this problem we define a restricted version of $V$. This is very similar to
the construction used in Avanzi, Heuberger and
Prodinger~\cite{Avanzi-Heuberger-Prodinger:2010:arith-of} and in Heuberger and
Krenn~\cite{Heuberger-Krenn:2010:wnaf-analysis}.


\begin{definition}[Restricted Voronoi Cell]\label{def:restr-voronoi}
  Let $V_u$ be a Voronoi cell with its centre $u$ as above. Let
  $v_0,\dots,v_{m-1}$ with appropriate $m\in\N$ be the vertices of $V_u$. We
  denote the midpoint of the line segment from $v_k$ to $v_{k+1}$ by
  $v_{k+1/2}$, and we use the convention that the indices are meant modulo $m$.

  The \emph{restricted Voronoi cell} $\wt{V}_u$ consists of
  \begin{itemize}
  \item the interior of $V_u$,
  \item the line segments from $v_{k+1/2}$ (excluded) to $v_{k+1}$ (excluded)
    for all $k$,
  \item the points $v_{k+1/2}$ for $k\in\set{0,\dots,\floor{\frac{m}{2}}-1}$,
    and
  \item the points $v_k$ for $k\in\set{1,\dots,\floor{\frac{m}{3}}}$.
  \end{itemize}
  Again we set $\wt{V} := \wt{V}_0$.
\end{definition}


\begin{figure}
  \centering
  \includegraphics{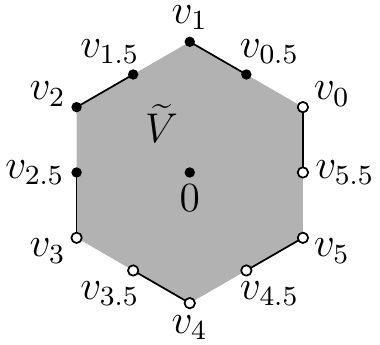}
  \caption[Restricted Voronoi cell $\wt{V}$ for $0$]{Restricted Voronoi cell
    $\wt{V}$ for $0$ corresponding to the set $\Ztau$ with $\tau =
    \frac{3}{2}+\frac{i}{2}\sqrt{3}$.}
  \label{fig:voronoi-restr}
\end{figure}


In Figure~\ref{fig:voronoi-restr} the restricted Voronoi cell of $0$ is shown
for $\tau=\frac{3}{2}+\frac{i}{2}\sqrt{3}$. The second condition in the
definition is used because it benefits symmetries. The third condition is just
to make the midpoints unique. Obviously, other rules\footnote{The rule has to
  make sure that the complex plane can be covered entirely and with no overlaps
  by restricted Voronoi cells, i.e.\ the condition $\C = \biguplus_{z\in\Ztau}
  \wt{V}_z$ has to be fulfilled.} could have been used to define the restricted
Voronoi cell.


The statements (including proofs) of the following lemma can be found in
Heuberger and Krenn~\cite{Heuberger-Krenn:2010:wnaf-analysis}. We use the
notation $\ballo{z}{r}$ for an open ball with centre~$z$ and radius~$r$ and
$\ballc{z}{r}$ for a closed ball.


\begin{lemma}[Properties of Voronoi Cells]\label{lem:voronoi-prop}
  We have the following properties:

  \begin{enumerate}[(a)]

  \item The vertices of $V$ are given explicitly by
    \begin{align*}
      v_0 &= 1/2 + \frac{i}{2\im{\tau}} 
      \left( \im{\tau}^2 + \fracpart{\re{\tau}}^2 
        - \fracpart{\re{\tau}} \right), \\
      v_1 &= \fracpart{\re{\tau}}-\frac12 + \frac{i}{2\im{\tau}} 
      \left( \im{\tau}^2 - \fracpart{\re{\tau}}^2 
        + \fracpart{\re{\tau}} \right), \\
      v_2 &= -1/2 + \frac{i}{2\im{\tau}}   
      \left( \im{\tau}^2 + \fracpart{\re{\tau}}^2 
        - \fracpart{\re{\tau}} \right)
      = v_0 - 1, \\
      v_3 &= -v_0, \\
      v_4 &= -v_1 \\
      \intertext{and}
      v_5 &= -v_2.
    \end{align*}
    All vertices have the same absolute value. If $\re{\tau}\in\Z$, then
    $v_1=v_2$ and $v_4=v_5$, i.e.\ the hexagon degenerates to a rectangle.

  \item The Voronoi cell $V$ is convex.

  \item We get $\ball*{0}{\textstyle\frac12} \subseteq V$.
    
  \item The inclusion $\tau^{-1} V \subseteq V$ holds.

  \end{enumerate}
  
\end{lemma}


\section{Digit Sets for Imaginary Quadratic Bases}
\label{sec:digit-sets-iq-bases}


In this section we assume that $\tau\in\C$ is an algebraic integer, imaginary
quadratic, i.e.\ $\tau$ is solution of an equation $\tau^2 - p \tau + q = 0$
with $p,q\in\Z$ and such that $q-p^2/4>0$. By $V$ we denote the Voronoi cell of
$0$ of the lattice $\Ztau$, by $\wt{V}$ the corresponding restricted Voronoi
cell, cf.\ Section~\ref{sec:voronoi}.


We consider \wNAF{} number systems
\begin{equation*}
  (\Ztau, z\mapsto \tau z, \cD),
\end{equation*}
where the digit set $\cD$ is the so called ``minimal norm representatives digit
set''. The following definition specifies that digit set, cf.\
Solinas~\cite{Solinas:1997:improved-algorithm,Solinas:2000:effic-koblit},
Blake, Kumar Murty and Xu~\cite{Blake-Kumar-Xu:2005:effic-algor} or Heuberger
and Krenn~\cite{Heuberger-Krenn:2010:wnaf-analysis}. It is used throughout this
article, whenever we have the setting (imaginary quadratic base) mentioned
above.


\begin{definition}[Minimal Norm Representatives Digit Set]
  \label{def:min-norm-digit-set}

  Let $w$ be an integer with $w\geq2$ and $\cD\subseteq\Ztau$ consist of $0$
  and exactly one representative of each residue class of $\Ztau$ modulo
  $\tau^w$ that is not divisible by $\tau$. If all such representatives
  $\eta\in\cD$ fulfil $\eta \in \tau^w \wt{V}$, then $\cD$ is called the
  \emph{minimal norm representatives digit set modulo $\tau^w$}.
\end{definition}


The previous definition uses the restricted Voronoi cell $\wt{V}$ for the point
$0$, see Definition~\ref{def:restr-voronoi}, to choose a representative with
minimal norm. Note that by construction of $\wt{V}$, there is only one such
choice for the digit set. Some examples of such digit sets are shown in
Figures~\ref{fig:digit-sets-pgeq3}, \ref{fig:digit-sets-pgeq3-p3q3},
\ref{fig:digit-sets-p2q2} and~\ref{fig:digit-sets-p0}.


\begin{figure}
  \centering 
  \subfloat[Digit set for $\tau=\frac32+\frac{i}{2}\sqrt{7}$ and $w=2$.]{
    \includegraphics[width=0.2\linewidth]{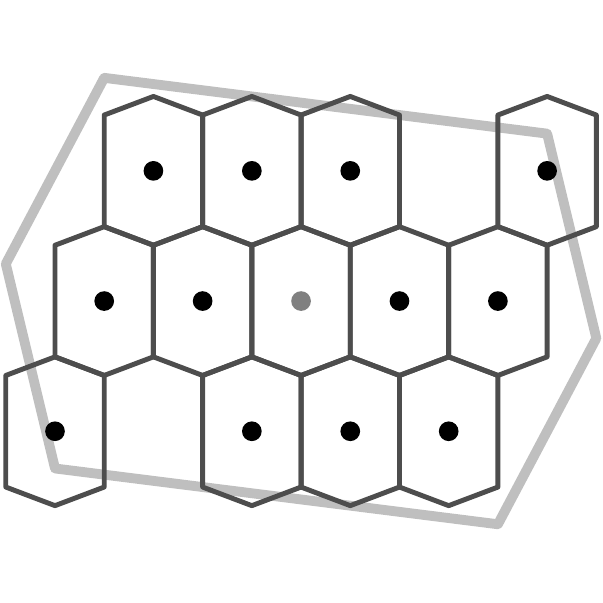}
    \label{fig:ds:p3q4w2}}
  \quad
  \subfloat[Digit set for $\tau=\frac32+\frac{i}{2}\sqrt{7}$ and $w=3$.]{
    \includegraphics[width=0.2\linewidth]{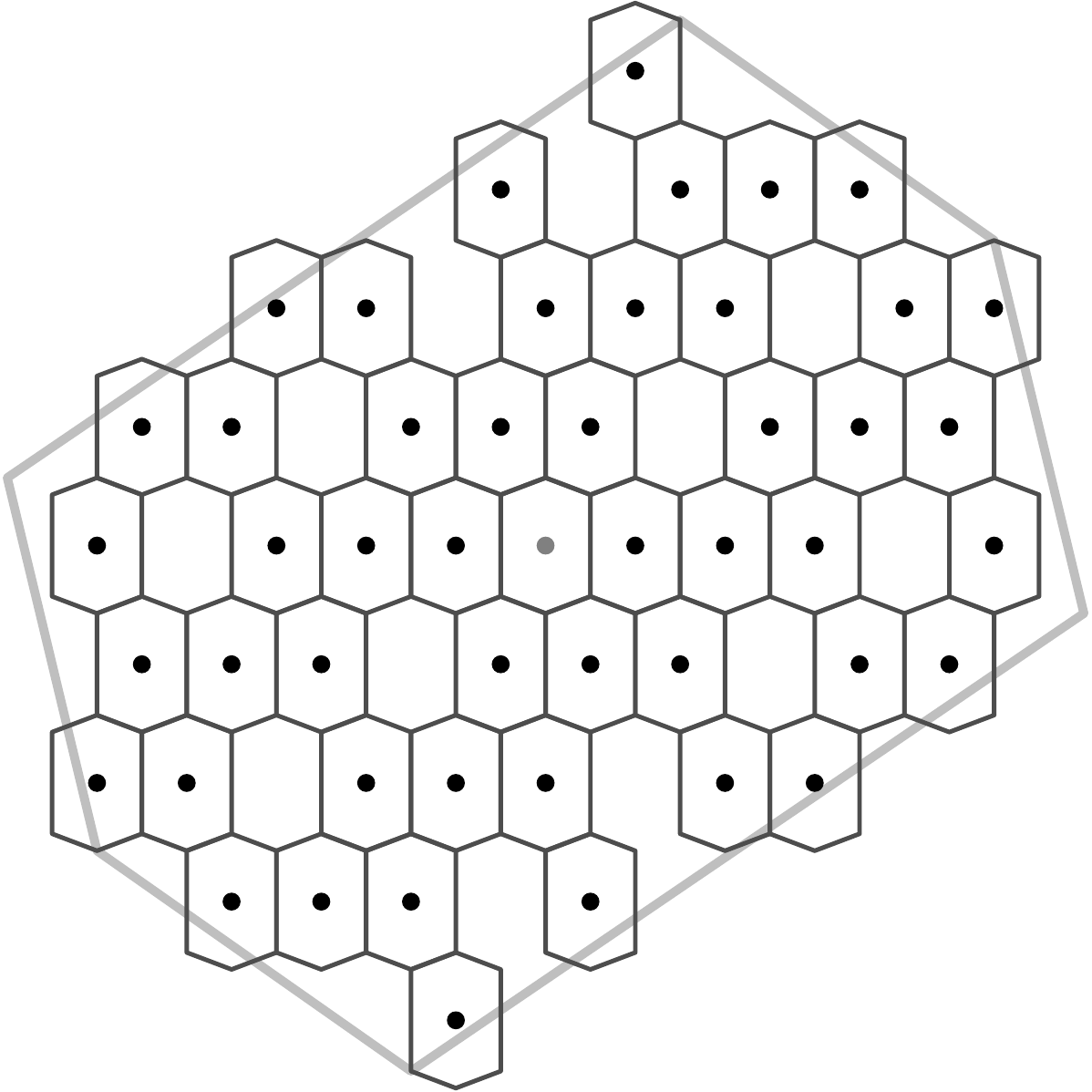}
    \label{fig:ds:p3q4w3}}
  \quad
  \subfloat[Digit set for $\tau=2+i$ and $w=2$.]{
    \includegraphics[width=0.2\linewidth]{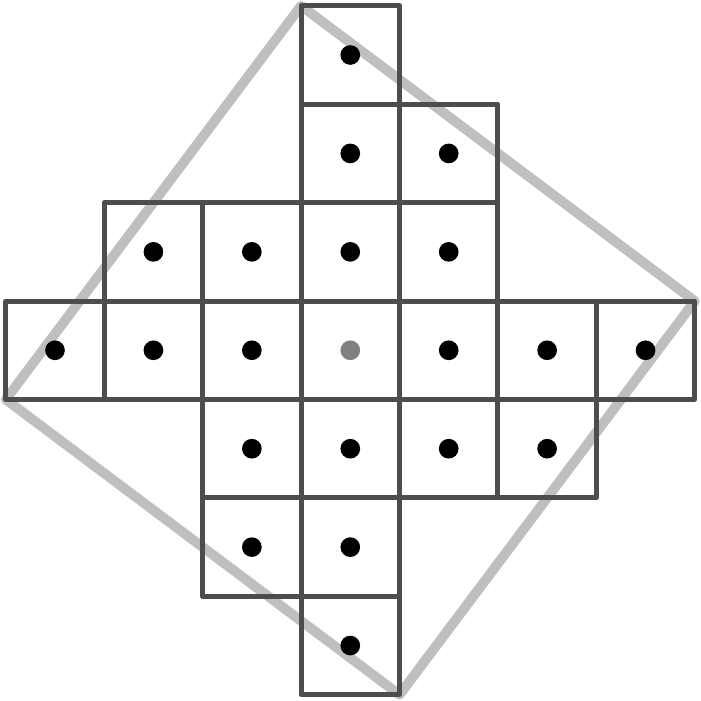}
    \label{fig:ds:p4q5w4}}
  \quad
  \subfloat[Digit set for $\tau=\frac52+\frac{i}{2}\sqrt{3}$ and $w=2$.]{
    \includegraphics[width=0.2\linewidth]{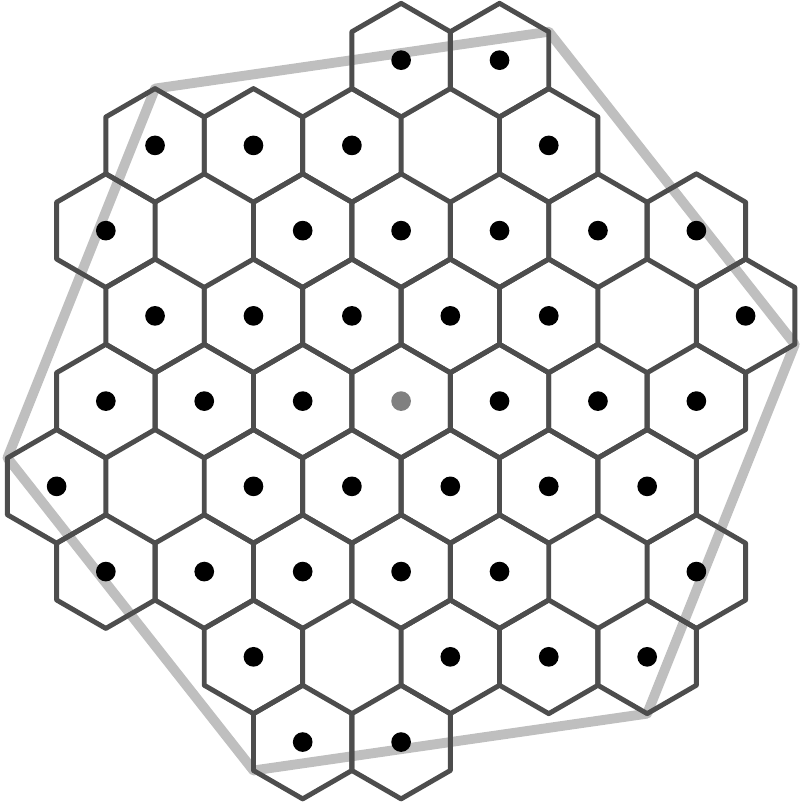}
    \label{fig:ds:p5q7w5}}
  \quad

  \caption[Digit sets for different $\tau$ and $w$]{Minimal norm representatives
    digit sets modulo $\tau^w$. For each digit
    $\eta$, the corresponding Voronoi cell $V_\eta$ is drawn. The large scaled
    Voronoi cell is $\tau^w V$.}
  \label{fig:digit-sets-pgeq3}
\end{figure}


\begin{remark}
  \label{rem:choice-digit-set-voronoi-boundary}

  The definition of a minimal norm representative digit set,
  Definition~\ref{def:min-norm-digit-set}, depends on the definition of the
  restricted Voronoi cell $\wt{V}$, Definition~\ref{def:restr-voronoi}. There
  we had some freedom in choosing which part of the boundary is included in
  $\wt{V}$, cf.\ the remarks after Definition~\ref{def:restr-voronoi}. We point
  out that all results given here for imaginary quadratic bases are valid for
  any admissible configuration of the restricted Voronoi cell, although only
  the case corresponding to Definition~\ref{def:restr-voronoi} will be
  presented.
\end{remark}


Using a minimal norm representatives digit set, each element of $\Ztau$
corresponds to a unique \wNAF{}, i.e.\ the pre-number system given at the
beginning of this section is indeed a \wNAF{} number system. This is stated in
the following theorem, which can be found in Heuberger and
Krenn~\cite{Heuberger-Krenn:2010:wnaf-analysis}.


\begin{theorem}[Existence and Uniqueness Theorem]
  \label{thm:existence-uniquness-wnaf}

  Let $w$ be an integer with $w\geq2$. Then the pre-number system
  \begin{equation*}
    (\Ztau, z\mapsto \tau z, \cD),
  \end{equation*}
  where $\cD$ is the minimal norm representatives digit set modulo $\tau^w$, is
  a non-redundant \wNAF{} number system, i.e.\ each lattice point $z\in\Ztau$
  has a unique \wNAF{}-expansion $\bfeta\in\cD^{\N_0}$ with
  $z=\expvalue{\bfeta}$.
\end{theorem}


\section{Optimality for Imaginary Quadratic Bases}
\label{sec:opt:i-q}


In this section we assume that $\tau\in\C$ is an algebraic integer, imaginary
quadratic, i.e.\ $\tau$ is solution of an equation $\tau^2 - p \tau + q = 0$
with $p,q\in\Z$ and such that $q-p^2/4>0$. Further let $w$ be an integer with
$w\geq2$ and let
\begin{equation*}
  (\Ztau, z\mapsto \tau z, \cD)
\end{equation*}
be the non-redundant \wNAF{} number system with minimal norm representatives
digit set modulo $\tau^w$, cf.\ Section~\ref{sec:digit-sets-iq-bases}.


Our main question in this section, as well as for the remaining part of this
article, is the following: For which bases and which $w$ is the width-$w$
non-adjacent form optimal? To answer this, we use the result from
Section~\ref{sec:abstract-optimality}. If we can show that the digit set $\cD$
is \wsubadditive{}, then optimality follows. This is done in the lemma
below. The result will then be formulated in
Corollary~\ref{cor:optimality-cor}, which, eventually, contains the optimality
result for our mentioned configuration.


\begin{lemma}\label{lem:opt-techn-condition}
  Suppose that one of the following conditions hold:
  \begin{enumerate}[(i)]
  \item \label{enu:tc-w4} $w\geq4$ and $\abs{p}\geq3$,
  \item \label{enu:tc-w3} $w=3$ and $\abs{p}\geq5$,
  \item \label{enu:tc-w3-p4} $w=3$, $\abs{p}=4$ and $5 \leq q \leq 9$,
  \item \label{enu:tc-w2-peven} $w=2$, $p$ even, and
    \begin{equation*}
      \left(\frac{1}{\sqrt{q}}+\frac{2}{q}\right)^2
      \left(q-\frac{p^2}{4}+1\right)
      < 1
    \end{equation*}
    or equivalently 
    \begin{equation*}
      \abs{p} > 2 \sqrt{q + 1 - \frac{q^2}{\left(2+\sqrt{q}\right)^2}},
    \end{equation*}
  \item \label{enu:tc-w2-podd} $w=2$, $p$ odd and
    \begin{equation*}
      \left(\frac{1}{\sqrt{q}}+\frac{2}{q}\right)^2
      \left(q-\frac{p^2}{4}+\frac{1}{4}\right)^2
      \left(q-\frac{p^2}{4}\right)^{-1}
      < 1.
    \end{equation*}
  \end{enumerate}
  Then the digit set $\cD$ is \wsubadditive{}.
\end{lemma}


\begin{figure}
  \centering
  \includegraphics{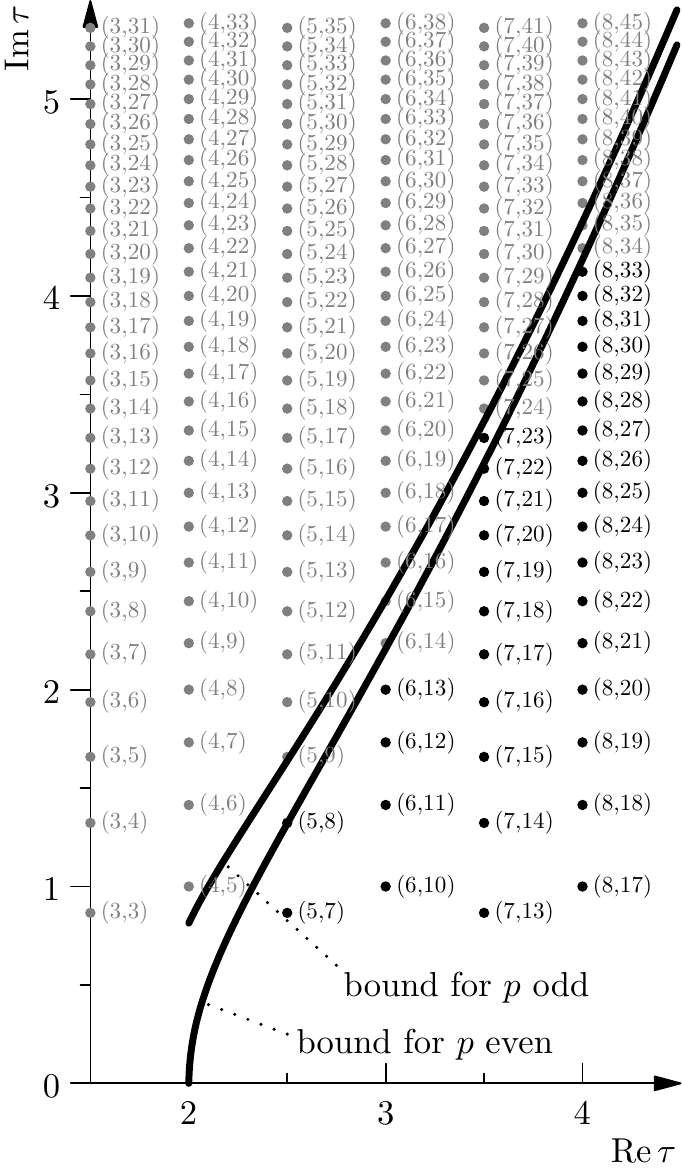}
  \caption{Bounds for the optimality of \wNAF[2]{}s. The two curves correspond
    to the conditions~\itemref{enu:tc-w2-peven} and~\itemref{enu:tc-w2-podd} of
    Lemma~\ref{lem:opt-techn-condition}. A dot corresponds to a valid
    $\tau$. If the dot is black, then the \wNAF[2]{}s of that $\tau$ are
    optimal, gray means not decidable with this method. Each dot is labelled
    with $(\abs{p},q)$.}
  \label{fig:bounds-w2}
\end{figure}


The conditions~\itemref{enu:tc-w2-peven} and~\itemref{enu:tc-w2-podd} of
Lemma~\ref{lem:opt-techn-condition}, i.e.\ the case $w=2$, are illustrated
graphically in Figure~\ref{fig:bounds-w2}.


\begin{proof} 
  If
  \begin{equation*}
    \tau^{w-1} V + V + V \subseteq \tau^w \interior{V}
  \end{equation*} 
  holds, then the digit set $\cD$ is \wsubadditive{} since $\cD \subseteq \tau^w
  V$, $-\cD \subseteq \tau^w V$, $V \subseteq \tau V$ and $z\in\tau^w
  \interior{V} \cap \Ztau$ implies that there is an integer $\ell\geq0$ with
  $z\in\tau^\ell\cD$. The sufficient condition of
  Proposition~\ref{pro:suff-cond-wsubadditive} was used with $U=\tau^w V$ and
  $S=\tau^w \interior{V} \setminus\set{0}$.

  Since $V$ is convex, it is sufficient to show that
  \begin{equation*}
    \tau^{w-1} V + 2V \subseteq \tau^w \interior{V}.
  \end{equation*}
  This will be done by showing
  \begin{equation*}
    \left(\abs\tau^{-1} + 2 \abs\tau^{-w}\right) \abs{V} < \tfrac12,
  \end{equation*}
  where $\abs{V}$ denotes the radius of the smallest closed disc with centre
  $0$ containing $V$. By setting
  \begin{equation*}
    \f{T}{p,q,w} := 2 \left(\abs\tau^{-1} + 2 \abs\tau^{-w}\right) \abs{V},
  \end{equation*}
  we have to show that
  \begin{equation*}
    \f{T}{p,q,w} < 1.
  \end{equation*}
  Remark that $\f{T}{p,q,w} > 0$, so it is sufficient to show
  \begin{equation*}
    \f{T^2}{p,q,w} < 1.
  \end{equation*}

  For each of the different conditions given, we will check that the inequality
  holds for special values of $p$, $q$ and $w$ and then use a monotonicity
  argument to get the result for other values of $p$, $q$ and $w$. In the
  following we distinguish between even and odd $p$.
  
  Let first $p$ be even, first. Then $\frac12 + \frac{i}{2} \im{\tau}$ is a
  vertex of the Voronoi cell $V$. This means $\abs{V} = \frac12
  \sqrt{1+q-p^2/4}$. Inserting that and $\abs\tau = \sqrt{q}$ in the asserted
  inequality yields
  \begin{equation*}
    \f{T^2}{p,q,w} = \left( \frac{1}{\sqrt{q}} + 2 q^{-w/2} \right)^2
    \left(1+q-\frac{p^2}{4}\right) < 1.
  \end{equation*}
  It is easy to see that the left hand side of this inequality is monotonically
  decreasing in $\abs{p}$ (as long as the condition $q>p^2/4$ is fulfilled)
  and monotonically decreasing in $w$. We assume $p\geq0$.

  If we set $p=4$ and $w=4$, we get
  \begin{equation*}
    \f{T^2}{4,q,4} = - \frac{12}{q^4} + \frac{4}{q^3} 
    - \frac{12}{q^{5/2}} + \frac{4}{q^{3/2}} - \frac{3}{q}+1,
  \end{equation*}
  which is strictly monotonically increasing for $q\geq5$. Further we get
  \begin{equation*}
    \lim_{q\to\infty} \f{T^2}{4,q,4} = 1.
  \end{equation*}
  This means $\f{T^2}{4,q,4} < 1$ for all $q\geq5$. Since $p\geq4$ implies
  $q\geq5$ and because of the monotonicity mentioned before, the
  case~\itemref{enu:tc-w4} for the even $p$ is completed.

  If we set $p=6$ and $w=3$, we get
  \begin{equation*}
    \f{T^2}{6,q,3} = -\frac{32}{q^3}-\frac{28}{q^2}-\frac{4}{q}+1,
  \end{equation*}
  which is obviously less than $1$. Therefore, again by monotonicity, the
  case~\itemref{enu:tc-w3} is done for the even $p$.

  If we set $p=4$ and $w=3$, we obtain
  \begin{equation*}
    \f{T^2}{4,q,3} = -\frac{12}{q^3}-\frac{8}{q^2}+\frac{1}{q}+1,
  \end{equation*}
  which is monotonically increasing for $5 \geq q \geq 18$. Further we get
  \begin{equation*}
    \f{T^2}{4,9,3} = \frac{242}{243} < 1.
  \end{equation*}
  This means $\f{T^2}{4,q,3} < 1$ for all $q$ with $5 \leq q \leq 9$. So
  case~\itemref{enu:tc-w3-p4} is completed.

  The condition given in \itemref{enu:tc-w2-peven} is exactly
  \begin{equation*}
    \f{T^2}{p,q,2} < 1
  \end{equation*}
  for even $p$, so the result follows immediately.

  Now, let $p$ be odd. Then $\frac{i}{2 \im{\tau}} \left( \im{\tau}^2 + \frac14
  \right)$ is a vertex of the Voronoi cell $V$. This means
  \begin{equation*}
    \abs{V} = \frac12 \left(q-\frac{p^2}{4}\right)^{-1/2} 
    \left( q - \frac{p^2}{4} + \frac14 \right).
  \end{equation*}
  Inserting that in the asserted inequality yields
  \begin{equation*}
    \f{T^2}{p,q,w} = \left(q-\frac{p^2}{4}\right)^{-1} 
    \left(q-\frac{p^2}{4}+\frac{1}{4}\right)^2 
    \left(2 q^{-w/2}+q^{-1/2}\right)^2 < 1.
  \end{equation*}
  Again, it is easy to verify that the left hand side of this inequality is
  monotonically decreasing in $p$ (as long as the condition $q \geq p^2/4 +
  1/4$ is fulfilled) and monotonically decreasing in $w$. We assume $p\geq0$.

  If we set $p=3$ and $w=4$, we get
  \begin{equation*}
    \f{T^2}{3,q,4} = \frac{4 (q-2)^2 \left(q^{3/2}+2\right)^2}{q^4 (4 q-9)}
  \end{equation*}
  which is strictly monotonically increasing for $q\geq3$. Further we get
  \begin{equation*}
    \lim_{q\to\infty} \f{T^2}{3,q,4} = 1.
  \end{equation*}
  This means $\f{T^2}{3,q,4} < 1$ for all $q\geq3$. Since $p\geq3$ implies
  $q\geq3$ and because of the monotonicity mentioned before, the
  case~\itemref{enu:tc-w4} for the odd $p$ is finished.

  If we set $p=5$ and $w=3$, we get
  \begin{equation*}
    \f{T^2}{5,q,3} = \frac{4 (q-6)^2 (q+2)^2}{q^3 (4 q-25)}
  \end{equation*}
  which is strictly monotonically increasing for $q\geq7$. Further we get
  \begin{equation*}
    \lim_{q\to\infty} \f{T^2}{5,q,3} = 1.
  \end{equation*}
  This means $0 < \f{T^2}{5,q,3} < 1$ for all $q\geq7$. As $p\geq5$
  implies $q\geq7$, using monotonicity again, the
  case~\itemref{enu:tc-w3} is done for the odd $p$.

  The condition given in \itemref{enu:tc-w2-podd} is exactly
  \begin{equation*}
    \f{T^2}{p,q,2} < 1
  \end{equation*}
  for odd $p$, so the result follows immediately.

  Since we have now analysed all the conditions, the proof is finished.
\end{proof}


Now we can prove the following optimality corollary, which is a consequence of
Theorem~\ref{thm:optimality-thm}.

 
\begin{corollary}\label{cor:optimality-cor}
  Suppose that one of the conditions~\itemref{enu:tc-w4}
  to~\itemref{enu:tc-w2-podd} of Lemma~\ref{lem:opt-techn-condition}
  holds. Then the width\nbd-$w$ non-adjacent form expansion for each element of
  $\Ztau$ is optimal.
\end{corollary}


\begin{proof}
  Lemma~\ref{lem:opt-techn-condition} implies that the digit set $\cD$ is
  \wsubadditive{}, therefore Theorem~\ref{thm:optimality-thm} can be used
  directly to get the desired result.
\end{proof}


\begin{figure}
  \centering
  \includegraphics{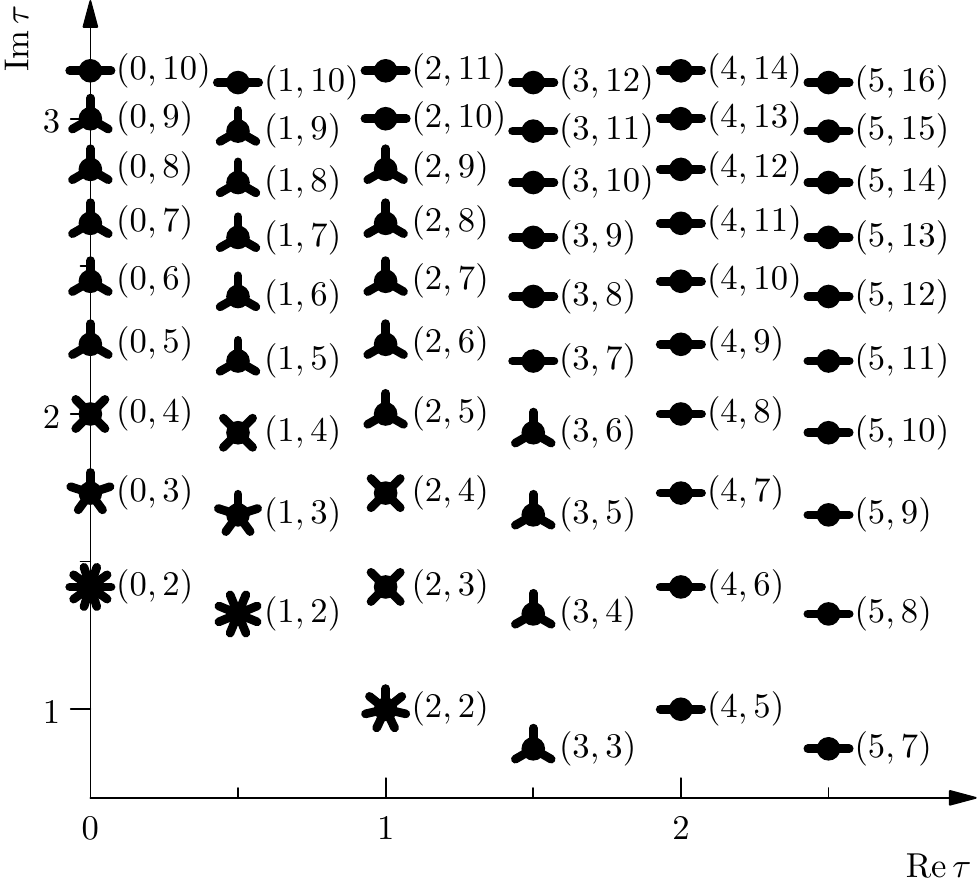}
  \caption[Voronoi cell $V$ for $0$]{Optimality of a
    \wNAF[(w-1)]{}-expansion. Each symbol is labelled with $(\abs{p},q)$. The
    number of lines around each symbol equals the minimal~$w$ for which there
    is an optimal \wNAF[(w-1)]{}-expansion of each element of $\Ztau$.}
  \label{fig:opt-w-1-NAF}
\end{figure}

\begin{remark}
  We have the following weaker optimality result. Let $p$, $q$ and $w$ be
  integers with $\abs{p} \geq p_0$, $q \geq q_0$ and $w \geq w_0$ for a
  $(p_0,q_o,w_0) \in L$, where
  \begin{multline*}
    L = \left\{ (0,10,2), (0,5,3),(0,4,4),(0,3,5),(0,2,10), \right. \\
    \left. (1,2,8),(2,3,4),(2,2,7),(3,7,2),(3,3,3),(4,5,2) \right\}.
  \end{multline*}
  Then we can show that the minimal norm representatives digit set modulo
  $\tau^w$ coming from a $\tau$ with $(p,q)$ is \wweaksubadditive{}, and
  therefore, by Remark~\ref{rem:wweaksubadditive}, we obtain optimality of a
  \wNAF[(w-1)]{}s of each element of $\Ztau$. The results are visualised
  graphically in Figure~\ref{fig:opt-w-1-NAF}.

  To show that the digit set is \wweaksubadditive{} we proceed in the same way
  as in the proof of Lemma~\ref{lem:opt-techn-condition}. We have to show the
  condition
  \begin{equation*}
    \f{T'}{p,q,w} < 1
  \end{equation*}
  where
  \begin{equation*}
    \f{T'}{p,q,w} = 2 \left(\abs\tau^{-2} + 2 \abs\tau^{-w}\right) \abs{V}
  \end{equation*}
  with $\abs\tau = \sqrt{q}$. When $p$ is even, we have
  \begin{equation*}
    \abs{V} = \frac12 \sqrt{1+q-\frac{p^2}{4}},
  \end{equation*}
  and when $p$ is odd, we have
  \begin{equation*}
    \abs{V} = \frac12 \left(q-\frac{p^2}{4}\right)^{-1/2} 
    \left( q - \frac{p^2}{4} + \frac14 \right).
  \end{equation*}
  Using monotonicity arguments as in the proof of
  Lemma~\ref{lem:opt-techn-condition} yields the list $L$ of ``critical
  points''.
\end{remark}


\section{The $p$-is-$3$-$q$-is-$3$-Case}
\label{sec:case-p-3-q-3}


One important case can be proved by using the Optimality Theorem of
Section~\ref{sec:abstract-optimality}, too, namely when $\tau$ comes from a
Koblitz curve in characteristic~$3$. We specialise the setting of
Section~\ref{sec:opt:i-q} to $p=3\mu$ with $\mu\in\set{-1,1}$ and $q=3$. We
continue looking at \wNAF{}-number systems with minimal norm representative
digit set modulo $\tau^w$ with $w\geq2$. Some examples of those digit sets are
shown in Figure~\ref{fig:digit-sets-pgeq3-p3q3}. We have the following
optimality result.


\begin{figure}
  \centering 
  \subfloat[Digit set for $\tau=\frac{3}{2} + \frac{i}{2}
  \sqrt{3}$ and $w=2$.]{
    \includegraphics[width=0.2\linewidth]{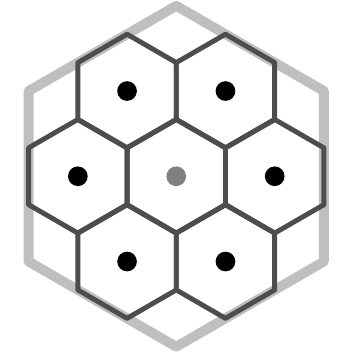}
    \label{fig:ds:p3q3w2}}
  \quad
  \subfloat[Digit set for $\tau=\frac{3}{2} + \frac{i}{2}
  \sqrt{3}$ and $w=3$.]{
    \includegraphics[width=0.2\linewidth]{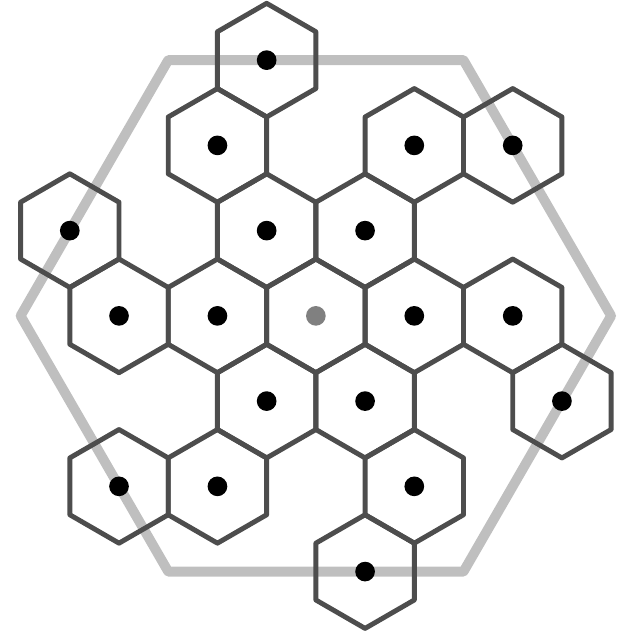}
    \label{fig:ds:p3q3w3}}
  \quad
  \subfloat[Digit set for $\tau=\frac{3}{2} + \frac{i}{2}
  \sqrt{3}$ and $w=4$.]{
    \includegraphics[width=0.2\linewidth]{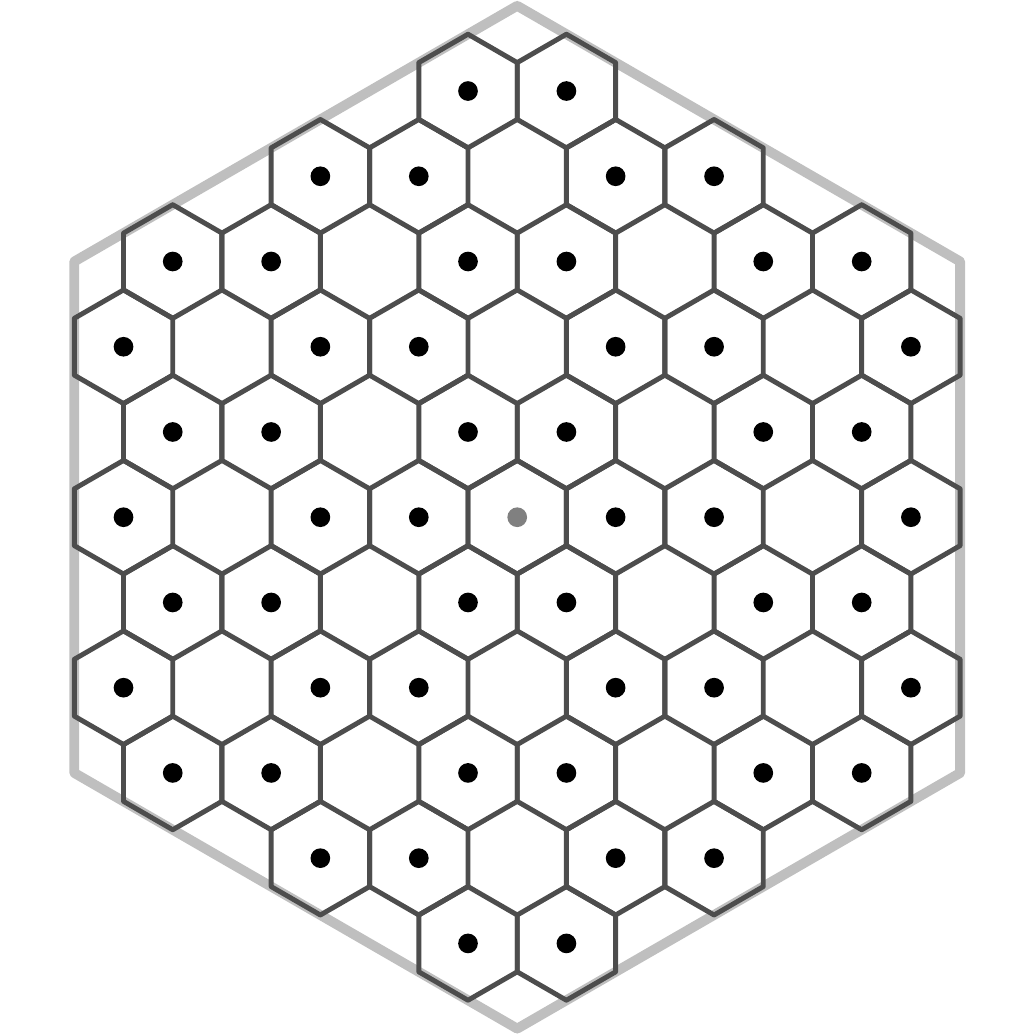}
    \label{fig:ds:p3q3w4}}

  \caption[Digit sets for different $\tau$ and $w$]{Minimal norm
    representatives digit sets modulo $\tau^w$. For each digit $\eta$, the
    corresponding Voronoi cell $V_\eta$ is drawn. The large scaled Voronoi cell
    is $\tau^w V$.}
  \label{fig:digit-sets-pgeq3-p3q3}
\end{figure}


\begin{corollary}\label{cor:optimality-koblitz3}
  With the setting above, the width\nbd-$w$ non-adjacent form expansion for
  each element of $\Ztau$ is optimal.
\end{corollary}


\begin{proof}
  Using the statement of Lemma~\ref{lem:opt-techn-condition} and
  Theorem~\ref{thm:optimality-thm} yields the optimality for all $w\geq4$. 

  Let $w=2$. Then our minimal norm representatives digit set is
  \begin{equation*}
    \cD = \set{0} \cup \bigcup_{0 \leq k < 6} \zeta^k \set{1},
  \end{equation*}
  where $\zeta$ is a primitive sixth root of unity, see Avanzi, Heuberger and
  Prodinger~\cite{Avanzi-Heuberger-Prodinger:2010:arith-of}. Therefore we
  obtain $\abs{\cD} = 1$ and $\cD = -\cD$. For $k\in\set{0,1}$ we get
  \begin{equation*}
    \tau^k\cD + \cD + \cD 
    \subseteq \ballc{0}{\sqrt{3}+2} 
    \subseteq \sqrt{3}^4 \ball{0}{\frac12} 
    \subseteq \tau^{2w} \interior{V},
  \end{equation*}
  so the digit set $\cD$ is \wsubadditive{} by the same arguments as in the
  beginning of the proof of Lemma~\ref{lem:opt-techn-condition}, and we can
  apply the Optimality Theorem to get the desired result.

  Let $w=3$. Then our minimal norm representatives digit set is
  \begin{equation*}
    \cD = \set{0} \cup \bigcup_{0 \leq k < 6} \zeta^k \set{1,2,4-\mu\tau},
  \end{equation*}
  where $\zeta$ is again a primitive sixth root of unity,
  again~\cite{Avanzi-Heuberger-Prodinger:2010:arith-of}. Therefore we obtain
  $\abs{\cD} = \abs{4-\mu\tau} = \sqrt{7}$ and again $\cD = -\cD$. For
  $k\in\set{0,1,2}$, we get
  \begin{equation*}
    \tau^k\cD + \cD + \cD 
    \subseteq \ballc{0}{5\sqrt{7}} 
    \subseteq \sqrt{3}^6 \ball{0}{\frac12} 
    \subseteq \tau^{2w} \interior{V},
  \end{equation*}
  so we can use Theorem~\ref{thm:optimality-thm} again to get the
  optimality.
\end{proof}


\section{The $p$-is-$2$-$q$-is-$2$-Case}
\label{sec:case-p-2-q-2}


In this section we look at another special base~$\tau$. We assume that
$p\in\set{-2,2}$ and $q=2$. Again, we continue looking at \wNAF{}-number
systems with minimal norm representative digit set modulo $\tau^w$ with
$w\geq2$. Some examples of those digit sets are shown in
Figure~\ref{fig:digit-sets-p2q2}.


\begin{figure}
  \centering 
  \subfloat[Digit set for $\tau=1+i$ and $w=2$.]{
    \includegraphics[width=0.2\linewidth]{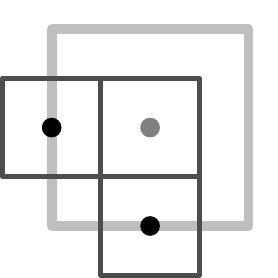}
    \label{fig:ds:p2q2w2}}
  \quad
  \subfloat[Digit set for $\tau=1+i$ and $w=3$.]{
    \includegraphics[width=0.2\linewidth]{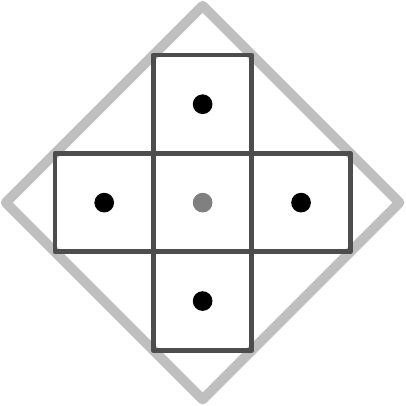}
    \label{fig:ds:p2q2w3}}
  \quad
  \subfloat[Digit set for $\tau=1+i$ and $w=4$.]{
    \includegraphics[width=0.2\linewidth]{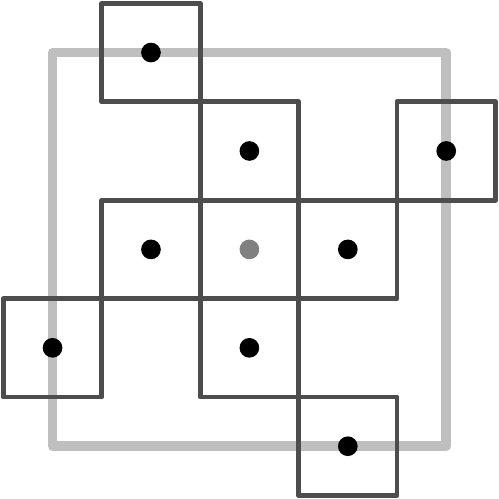}
    \label{fig:ds:p2q2w4}}
  \quad
  \subfloat[Digit set for $\tau=1+i$ and $w=5$.]{
    \includegraphics[width=0.2\linewidth]{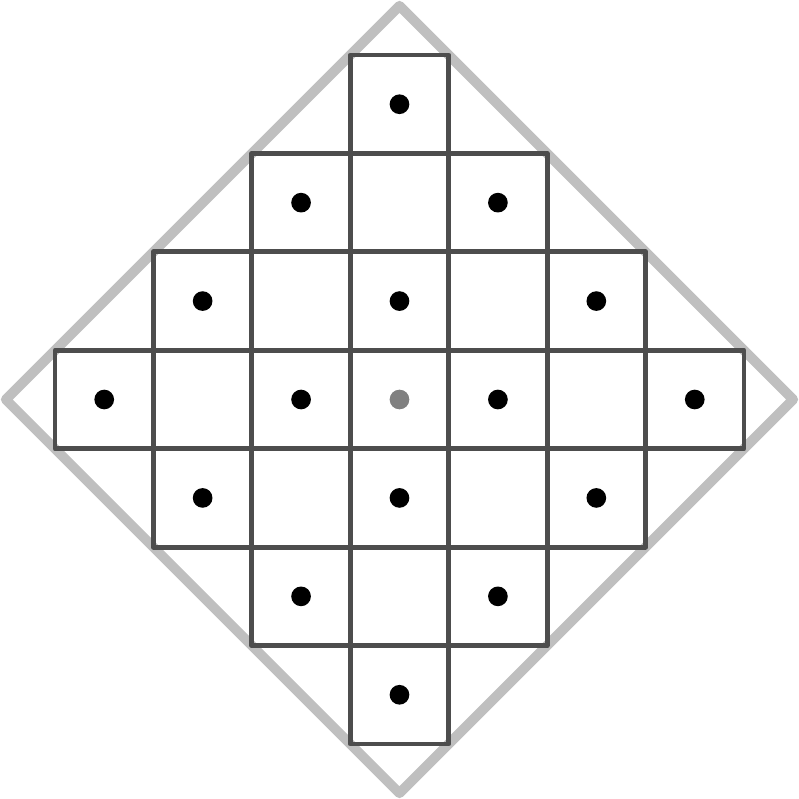}
    \label{fig:ds:p2q2w5}}
  \quad

  \caption[Digit sets for different $\tau$ and $w$]{Minimal norm
    representatives digit sets modulo $\tau^w$. For each digit $\eta$, the
    corresponding Voronoi cell $V_\eta$ is drawn. The large scaled Voronoi cell
    is $\tau^w V$.}
  \label{fig:digit-sets-p2q2}
\end{figure}


For all possible $\tau$ of this section, the corresponding Voronoi cell can be
written explicitly as
\begin{equation*}
  V = \f{\mathsf{polygon}}{\set{\tfrac12(1+i), \tfrac12(-1+i),
      \tfrac12(-1-i), \tfrac12(1-i)}}.
\end{equation*}
Remark that $V$ is an axis-parallel square and that we have
\begin{equation*}
  \tau V = \f{\mathsf{polygon}}{\set*{i^j}{j\in\set{0,1,2,3}}}. 
\end{equation*}

In this section we will prove that the \wNAF{}s are optimal if and only if $w$
is odd. The first part, optimality for odd $w$, is written down as the theorem
below. The non-optimality part for even $w$ can be found as
Proposition~\ref{pro:p-2-q-2-w-even-non-opt}.


\begin{theorem}\label{thm:p-2-q-2-w-odd-optimal}
  Let $w$ be an odd integer with $w\geq3$, and let $z\in\Ztau$. Then the
  width\nbd-$w$ non-adjacent form expansion of $z$ is optimal.
\end{theorem}


\begin{remark}
  Let $w$ be an odd integer with $w\geq3$. Let $z \in \tau^w V \cap \Ztau$,
  then $z$ can be represented as a \wNAF{} expansion with weight at most
  $1$. To see this, consider the boundary of $\tau^wV$. Its vertices are
  $2^{(w-1)/2} i^m$ for $m\in\set{0,1,2,3}$. All elements of
  $\boundary{\tau^wV}\cap\Ztau$ can be written as $2^{(w-1)/2} i^m + k(1+i)i^n$
  for some integers $k$, $m$ and $n$. Further, all those elements are
  divisible by $\tau$. Therefore each digit lies in the interior of $\tau^wV$,
  and for each $z \in \tau^w V \cap \Ztau$ there is an integer $\ell\geq0$ such
  that $\tau^{-\ell}z\in\cD$, because $\tau^{-1}V\subseteq V$ and
  $\abs{\tau}>1$.
\end{remark}


\begin{proof}[Proof of Theorem~\ref{thm:p-2-q-2-w-odd-optimal}]
  We prove that the digit set~$\cD$ is \wsubadditive{}. Hence, optimality
  follows using Theorem~\ref{thm:optimality-thm}. Using the remark above,
  $\cD=-\cD$ and the ideas of Proposition~\ref{pro:suff-cond-wsubadditive}, it
  is sufficient to show
  \begin{equation*}
    \tau^{-w} \left( \tau^k\cD + \cD + \cD \right) \cap \Ztau
    \subseteq \tau^wV
  \end{equation*}
  for $k\in\set{0,\dots,w-1}$.

  Let $k=w-1$. We show that
  \begin{equation}\label{eq:p-2-q-2:inclusion}
    \left( \cD + \tau^{-(w-1)} \left(\cD + \cD\right) \right) \cap \tau\Ztau
    \subseteq \tau^{w+1}V.
  \end{equation}
  So let $y=b+a$ be an element of the left hand side
  of~\eqref{eq:p-2-q-2:inclusion} with $b\in\cD$ and $a\in\tau^{-(w-1)}
  \left(\cD + \cD\right)$. We can assume $y\neq0$. Since $y\in\Ztau$ and $\cD
  \subseteq \Ztau$, we have $a\in\Ztau$. Since $\cD\subseteq \tau^wV$, we
  obtain
  \begin{equation*}
    \tau^{-(w-1)} \left(\cD + \cD\right) \subseteq 2 \tau V.
  \end{equation*}
  Because $2\tau V = \tau^3V \subseteq \tau^wV$, we can assume $b\neq0$. This
  means $\tau\ndivides b$. Since $\tau \divides y$, we have $\tau\ndivides a$.
  The set $2 \tau V \cap \Ztau$ consists exactly of $0$, $i^m$, $2i^m$ and
  $\tau i^m$ for $m\in\set{0,1,2,3}$. The only elements in that set not
  divisible by $\tau$ are the $i^m$. Therefore $a=i^m$ for some $m$. The digit
  $b$ is in the interior of $\tau^wV$, thus $y=b+a$ is in $\tau^wV \subseteq
  \tau^{w+1}V$.

  Now let $k\in\set{0,\dots,w-2}$. If $w\geq5$, then
  \begin{equation*}
    \tau^{-w} \left( \tau^k\cD + \cD + \cD \right) 
    \subseteq \tau^{w-2}V + 2V,
  \end{equation*}
  using $\cD\subseteq \tau^wV$ and properties of the Voronoi cell $V$.
  Consider the two squares $\tau^{w-2}V$ and $\tau^wV = 2\tau^{w-2}V$. The
  distance between the boundaries of them is at least $\tfrac12
  \abs\tau^{w-2}$, which is at least $\sqrt{2}$. Since $2V$ is contained in a
  disc with radius $\sqrt{2}$, we obtain $\tau^{w-2}V + 2V \subseteq \tau^wV$.

  We are left with the case $w=3$ and $k\in\set{0,1}$. There the digit set $\cD$
  consists of $0$ and $i^m$ for $m\in\set{0,1,2,3}$. Therefore we have
  $\cD\subseteq\tau V$ (instead of $\cD\subseteq\tau^3 V$). By the same
  arguments as in the previous paragraph we get
  \begin{equation*}
    \tau^{-3} \left( \tau^k\cD + \cD + \cD \right) 
    \subseteq \tfrac12 \left( \tau V + 2V \right) 
    \subseteq \tau^3 V,
  \end{equation*}
  so the proof is complete.
\end{proof}


The next result is the non-optimality result for even $w$. 


\begin{proposition}\label{pro:p-2-q-2-w-even-non-opt}
  Let $w$ be an even integer with $w\geq2$. Then there is an element of $\Ztau$
  whose \wNAF{}-expansion is non-optimal.
\end{proposition}


Again, some examples of the digit sets used are shown in
Figure~\ref{fig:digit-sets-p2q2}. The proof of the proposition is split up:
Lemma~\ref{lem:p-2-q-2-w-even} handles the general case for even $w\geq4$ and
Lemma~\ref{lem:p-2-q-2-w-2} gives a counter-example (to optimality) for $w=2$.

For the remaining section---it contains the proof of
Proposition~\ref{pro:p-2-q-2-w-even-non-opt}---we will assume $\tau=1+i$. All
other cases are analogous.


\begin{lemma}\label{lem:p-2-q-2-w-even}
  Let the assumptions of Proposition~\ref{pro:p-2-q-2-w-even-non-opt} hold and
  suppose $w\geq4$. Define $A := \abs\tau^w\frac12(1-i)$ and $B := \frac1\tau
  A$ and set $s=-i^{1-w/2}$. Then
  \begin{enumerate}[(a)]
  \item $1$, $i$, $-1$ and $-i$ are digits,
  \item $A-1$ is a digit,
  \item $-B-1$ is a digit,
  \item $i\tau^{w-1}-s^{-1}$ is a digit, and
  \item we have
    \begin{equation*}
      (A-1) \tau^{w-1} + (-s^{-1}) 
      = s\tau^{2w} + (-B-1)\tau^w + (i\tau^{w-1}-s^{-1}).
    \end{equation*}
  \end{enumerate}
\end{lemma}
 

\begin{figure}
  \centering
  \includegraphics{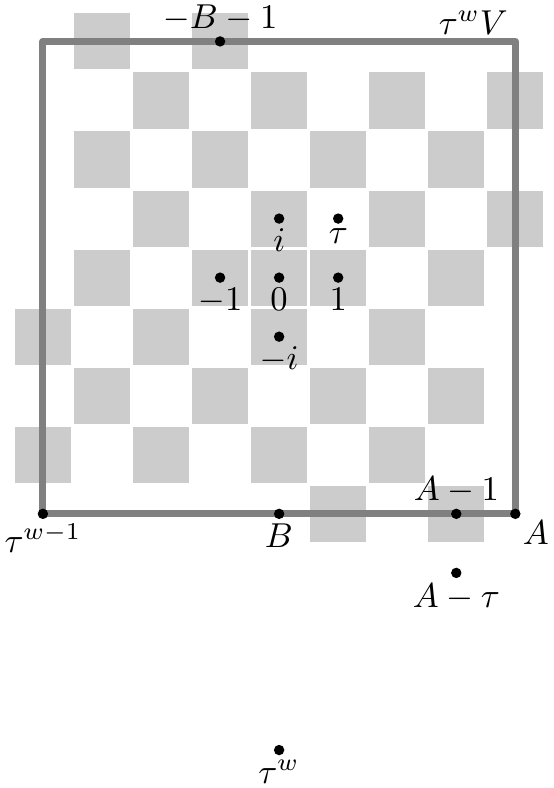}
  \caption{The $w$-is-even situation. The figure shows the configuration
    $p=2$, $q=2$, $w=6$, $s=1$. A polygon filled grey represents a digit, a dot
    represents a point of interest in Lemma~\ref{lem:p-2-q-2-w-even}.}
  \label{fig:nonopt-p2q2weven}
\end{figure}

Figure~\ref{fig:nonopt-p2q2weven} shows the digits used in
Lemma~\ref{lem:p-2-q-2-w-even} for a special configuration. 


\begin{proof}
  \begin{enumerate}[(a)]

  \item A direct calculation shows that the lattice elements $1$, $i$, $-1$
    and $-i$ are in the interior of
    \begin{equation*}
      \tau^wV 
      = \intervalcc{-2^{w/2-1}}{2^{w/2-1}} + \intervalcc{-2^{w/2-1}}{2^{w/2-1}} i
    \end{equation*}
    and are not divisible by $\tau$. So all of them are digits.

  \item We can rewrite $A$ as
    \begin{equation*}
      A = 2^{w/2-1} (1-i) = - 2^{w/2-1} i \tau,
    \end{equation*}
    therefore $\tau^2 \divides A$. We remark that $A$ is a vertex (the
    lower-right vertex) of the scaled Voronoi cell $\tau^wV$ and that the edges
    of $\tau^wV$ are parallel to the real and imaginary axes. This means that
    $A-1$ is on the boundary, too, and its real part is larger than $0$. By
    using the construction of the restricted Voronoi cell, cf.\
    Definition~\ref{def:restr-voronoi}, we know that $A-1$ is in
    $\tau^w\wt{V}$. Since it is clearly not divisible by $\tau$, it is a digit.

  \item We have
    \begin{equation*}
      B = \tfrac1\tau A = - 2^{w/2-1} i.
    \end{equation*}
    Therefore $\tau \divides B$, and we know that $B$ halves the edge at the
    bottom of the Voronoi cell $\tau^wV$. By construction of the scaled
    restricted Voronoi cell $\tau^w\wt{V}$, cf.\
    Definition~\ref{def:restr-voronoi}, we obtain that $B+1$ is a digit, and
    therefore, by symmetry, $-B-1$ is a digit, too.

  \item Rewriting yields
    \begin{equation*}
      i\tau^{w-1}-s^{-1} = s^{-1} (is\tau^{w-1}-1),
    \end{equation*}
    and we obtain
    \begin{equation*}
      s\tau^w 
      = -i^{1-w/2} (1+i)^w
      = - 2^{w/2} i,
    \end{equation*}
    since $(1+i)^2=2i$. Further we can check that the vertices of $\tau^wV$ are
    $i^k \tau^{w-1}$ for an appropriate $k\in\Z$. 

    Now consider $is\tau^{w-1}$. This is exactly the lower-right vertex $A$ of
    $\tau^wV$. Therefore, we have
    \begin{equation*}
      i\tau^{w-1}-s^{-1} = s^{-1} (A-1).
    \end{equation*}
    Using that $A-1$ is a digit and the rotational symmetry of the restricted
    Voronoi cell, $i\tau^{w-1}-s^{-1}$ is a digit.

  \item As before, we remark that $s\tau^w = -2^{w/2}i$. Therefore we obtain
    \begin{equation*}
      B-1-s\tau^w = -B-1.
    \end{equation*}
    Now, by rewriting, we get
    \begin{align*}
      (A-1) \tau^{w-1} + (-s^{-1})
      &= (A-\tau) \tau^{w-1} + (i\tau^{w-1}-s^{-1}) \\
      &= (B-1) \tau^w + (i\tau^{w-1}-s^{-1}) \\
      &= s\tau^{2w} + (B-1-s\tau^w) \tau^w + (i\tau^{w-1}-s^{-1}) \\
      &= s\tau^{2w} + (-B-1)\tau^w + (i\tau^{w-1}-s^{-1}),
    \end{align*}
    which was to prove.
    \qedhere
  \end{enumerate}
\end{proof}


\begin{lemma}\label{lem:p-2-q-2-w-2}
  Let the assumptions of Proposition~\ref{pro:p-2-q-2-w-even-non-opt} hold and
  suppose $w=2$. Then
  \begin{enumerate}[(a)]
  \item $-1$ and $-i$ are digits and
  \item we have
    \begin{equation*}
      - \tau - 1 = - i \tau^6 - \tau^4  - i \tau^2 - i.
    \end{equation*}
  \end{enumerate}
\end{lemma}


\begin{proof}
  \begin{enumerate}[(a)]

  \item The elements $-1$ and $-i$ are on the boundary of the Voronoi cell
    $\tau^2V$, cf.\
    Figure~\ref{fig:digit-sets-p2q2}\subref{fig:ds:p2q2w2}. More precisely,
    each is halving an edge of the Voronoi cell mentioned. The construction of
    the restricted Voronoi cell, together with the rotation and scaling of
    $\tau^2=2i$, implies that $-1$ and $-i$ are in $\tau^2\wt{V}$. Since none of
    them is divisible by $\tau$, both are digits.

  \item The element $i$ has the \wNAF[2]{}-representation
    \begin{equation*}
      i = -i\tau^4 - \tau^2  - i.
    \end{equation*}
    Therefore we obtain
    \begin{equation*}
      - \tau - 1 
      = (-1+i) \tau + (i \tau - 1)
      = i \tau^2 + (-i)
      = -i\tau^6 - \tau^4  - i \tau^2 - i,
    \end{equation*}
    which was to show.
    \qedhere
  \end{enumerate}
\end{proof}


Finally, we are able to prove the non-optimality result.


\begin{proof}[Proof of Proposition~\ref{pro:p-2-q-2-w-even-non-opt}]
  Let $w\geq4$. Everything needed can be found in
  Lemma~\ref{lem:p-2-q-2-w-even}: We have the equation
  \begin{equation*}
    (A-1) \tau^{w-1} + (-s^{-1}) 
    = s\tau^{2w} + (-B-1)\tau^w + (i\tau^{w-1}-s^{-1}),
  \end{equation*}
  in which the left and the right hand side are both valid expansion (the
  coefficients are digits). The left hand side has weight~$2$ and is not a
  \wNAF{}, whereas the right hand side has weight~$3$ and is a \wNAF{}.

  Similarly the case $w=2$ is shown in Lemma~\ref{lem:p-2-q-2-w-2}: We have the
  equation
  \begin{equation*}
    - \tau - 1 = - i \tau^6 - \tau^4  - i \tau^2 - i,
  \end{equation*}
  which again is a counter-example to the optimality of the \wNAF[2]{}s.
\end{proof}


\section{The $p$-is-$0$-Case}
\label{sec:case-p-0}


This section contains another special base~$\tau$. We assume that $p=0$ and
that we have an integer $q\geq2$. Again, we continue looking at \wNAF{}-number
systems with minimal norm representative digit set modulo $\tau^w$ with
$w\geq2$. Some examples of the digit sets used are shown in
Figure~\ref{fig:digit-sets-p0}.


\begin{figure}
  \centering 
  \subfloat[Digit set for $\tau=i\sqrt{2}$ and $w=3$.]{
    \includegraphics[width=0.2\linewidth]{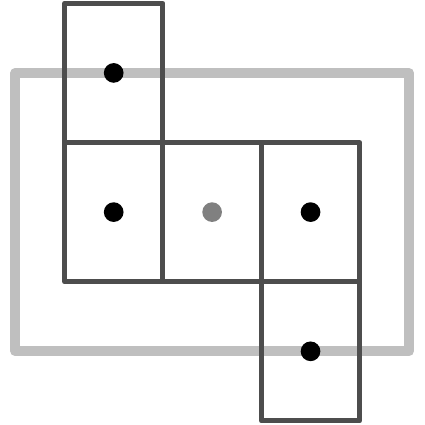}
    \label{fig:ds:p0q2w3}}
  \quad
  \subfloat[Digit set for $\tau=i\sqrt{2}$ and $w=5$.]{
    \includegraphics[width=0.2\linewidth]{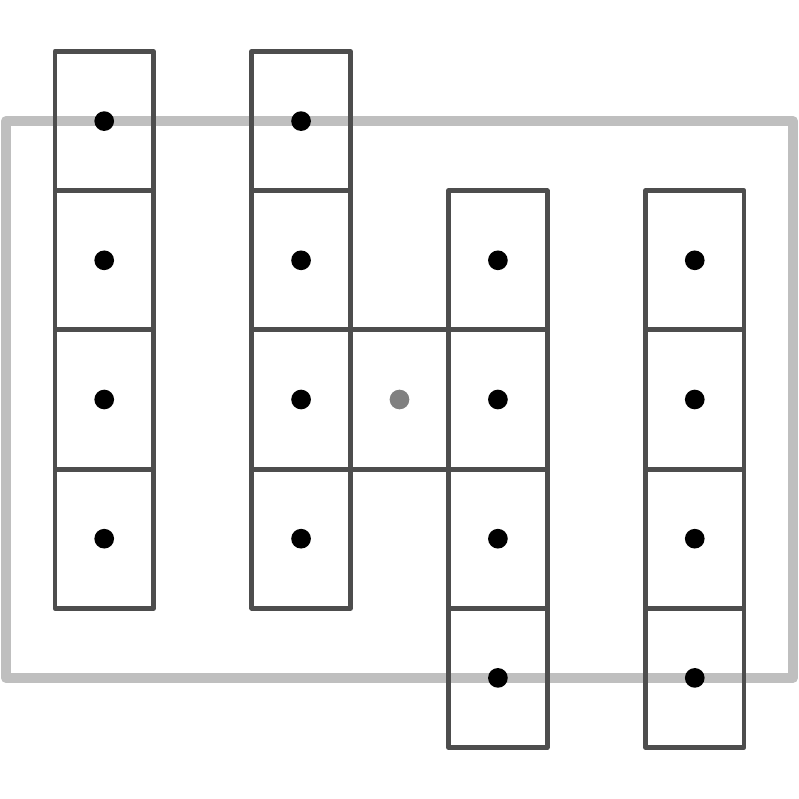}
    \label{fig:ds:p0q2w5}}
  \quad
  \subfloat[Digit set for $\tau=i\sqrt{3}$ and $w=3$.]{
    \includegraphics[width=0.2\linewidth]{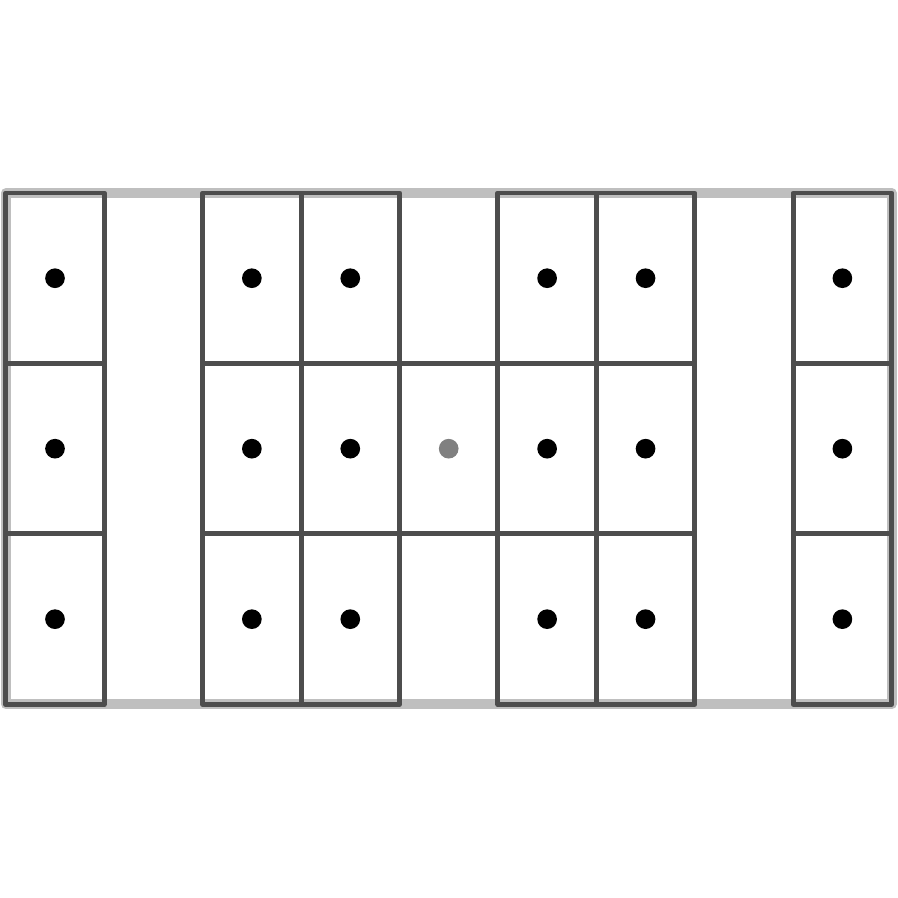}
    \label{fig:ds:p0q3w3}}
  \quad
  \subfloat[Digit set for $\tau=i\sqrt{3}$ and $w=5$.]{
    \includegraphics[width=0.2\linewidth]{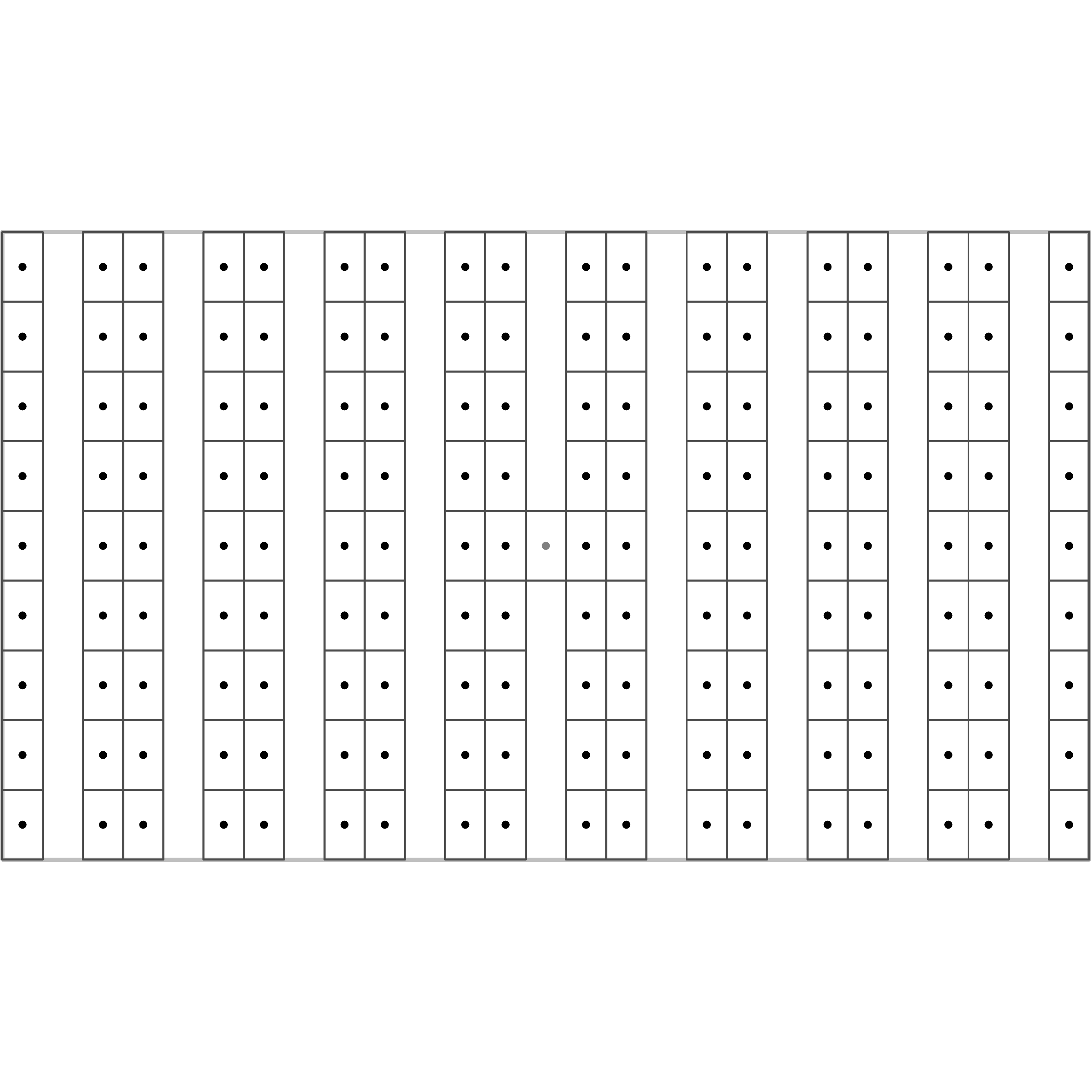}
    \label{fig:ds:p0q3w5}}
  \quad

  \caption[Digit sets for different $\tau$ and $w$]{Minimal norm
    representatives digit sets modulo $\tau^w$. For each digit $\eta$, the
    corresponding Voronoi cell $V_\eta$ is drawn. The large scaled Voronoi cell
    is $\tau^w V$.}
  \label{fig:digit-sets-p0}
\end{figure}


For all possible $\tau$ of this section, the corresponding Voronoi cell can be
written explicitly as
\begin{equation*}
  V = \f{\mathsf{polygon}}{\set{\tfrac12(\tau+1), \tfrac12(\tau-1),
      \tfrac12(-\tau-1), \tfrac12(-\tau+1)}}.
\end{equation*}
Remark that $V$ is an axis-parallel rectangle.


In this section we prove the following non-optimality result.

\begin{proposition}\label{pro:p-0-non-opt}
  Let $w$ be an odd integer with $w\geq3$ and the setting as above. Then there
  is an element of $\Ztau$ whose \wNAF{}-expansion is non-optimal.
\end{proposition}


For the remaining section---it contains the proof of the proposition above---we will assume $\tau=i\sqrt{q}$. The case $\tau=-i\sqrt{q}$ is
analogous. Before we start with the proof of
Proposition~\ref{pro:p-0-non-opt}, we need the following two lemmata.


\begin{lemma}\label{lem:p-0-q-even}
  Let the assumptions of Proposition~\ref{pro:p-0-non-opt} hold, and suppose
  that $q$ is even. Define $A := \frac12 \abs\tau^{w+1}$ and $B := \frac1\tau
  A$, and set $s=(-1)^{\frac12(w+1)}$. Then
  \begin{enumerate}[(a)]
  \item $1$ and $-1$ are digits,
  \item $A-1-\tau$ is a digit,
  \item $-B-1$ is a digit,
  \item $-s-\tau^{w-1}$ is a digit, and
  \item we have
    \begin{equation*}
      (A-1-\tau) \tau^{w-1} - s = s \tau^{2w} + (-B-1)\tau^w + (-s-\tau^{w-1}).
    \end{equation*}
  \end{enumerate}
\end{lemma}


\begin{figure}
  \centering
  \includegraphics{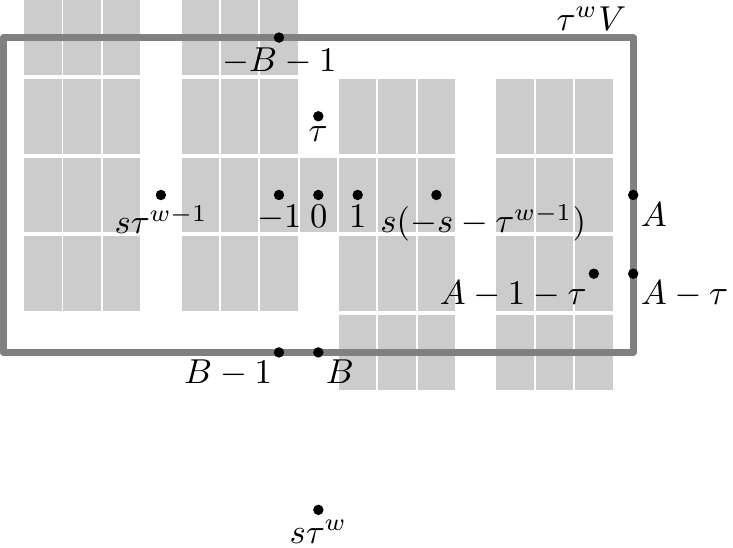}
  \caption{The $q$-is-even situation. The figure shows the configuration $p=0$,
    $q=4$, $\tau=2i$, $w=3$, $s=1$. A polygon filled grey represents a digit, a
    dot represents a point of interest in Lemma~\ref{lem:p-0-q-even}.}
  \label{fig:nonopt-p0qeven}
\end{figure}


Figure~\ref{fig:nonopt-p0qeven} shows the digits used in
Lemma~\ref{lem:p-0-q-even} for a special configuration. 


\begin{proof}
  
  \begin{enumerate}[(a)]

  \item A direct calculation shows that $-1$ and $1$ are in an open disc with
    radius $\frac12 \abs\tau^w$, which itself is contained in $\tau^wV$. Both
    are not divisible by $\tau$, so both are digits.

  \item Because $w$ is odd, $q$ is even and $\tau=i\sqrt{q}$, we can rewrite
    the point $A$ as
    \begin{equation*}
      A = \frac12 \abs\tau^{w+1} = \frac{q}{2} q^{\frac12 (w-1)}
    \end{equation*}
    and see that $A$ is a (positive) rational integer and that $\tau^{w-1}
    \divides A$. Furthermore, $A$ halves an edge of $\tau^wV$. Therefore, $A-1$
    is inside $\tau^wV$. If $q\geq4$ or $w\geq5$, the point $A-1-\tau$ is
    inside $\tau^wV$, too, since the vertical (parallel to the imaginary axis)
    side-length of $\tau^wV$ is $\abs\tau^w$ and $\abs\tau < \frac12
    \abs\tau^w$. Since $\tau^2 \divides A$, we obtain $\tau \ndivides
    A-1-\tau$, so $A-1-\tau$ is a digit. If $q=2$ and $w=3$, we have
    $A-1-\tau=1-\tau$. Due to the definition of the restricted Voronoi cell
    $\wt{V}$, cf.\ Definition~\ref{def:restr-voronoi}, we obtain that $1-\tau$
    is a digit.

  \item Previously we saw $\tau^{w-1} \divides A$. Using the definition of $B$
    and $w\geq3$ yields $\tau \divides B$. It is easy to check that $B=\frac12
    s\tau^w$. Furthermore, we see that $B$ is on the boundary of the
    Voronoi cell $\tau^wV$. By a symmetry argument we get the same results for
    $-B$. By the construction of the restricted Voronoi cell $\wt{V}$, cf.\
    Definition~\ref{def:restr-voronoi}, we obtain that $-B-1$ is in
    $\tau^w\wt{V}$ and since clearly $\tau \ndivides (-B-1)$, we get that
    $-B-1$ is a digit.

  \item We first remark that $\tau^{w-1} \in \Z$ and that $\abs{\tau^{w-1}}
    \leq A$. Even more, we get $0 < -s\tau^{w-1} \leq A$. Since $A$ is on the
    boundary of $\tau^wV$, we obtain $-1-s\tau^{w-1} \in
    \tau^w\interior{V}$. By symmetry the result is true for $-s-\tau^{w-1}$ and
    clearly $\tau \ndivides (-s-\tau^{w-1})$, so $-s-\tau^{w-1}$ is a digit.

  \item We get
    \begin{align*}
      (A-1-\tau) \tau^{w-1} + (-s)
      &= (A-\tau) \tau^{w-1} + (-s-\tau^{w-1}) \\
      &= (B-1) \tau^w + (-s-\tau^{w-1}) \\
      &= s \tau^{2w} + (B-1-s\tau^w) \tau^w + (-s-\tau^{w-1}) \\
      &= s \tau^{2w} + (-B-1)\tau^w + (-s-\tau^{w-1}),
    \end{align*}
    which can easily be verified. We used $B=\frac1\tau A$. 
    \qedhere
  \end{enumerate}
\end{proof}


\begin{lemma}\label{lem:p-0-q-odd}
  Let the assumptions of Proposition~\ref{pro:p-0-non-opt} hold, and suppose
  that $q$ is odd. Define $A' := \frac12 \abs\tau^{w+1}$, $B' := \frac1\tau
  A$, $A:=A'-\frac12$ and $B:=B'+\frac\tau2$, and set $C=-A$, $t =
  (q+1)/2$ and $s=(-1)^{\frac12(w+1)} \in\set{-1,1}$. Then
  \begin{enumerate}[(a)]
  \item $1$ and $-1$ are digits,
  \item $A-\tau$ is a digit,
  \item $sC$ is a digit,
  \item $-B-1$ is a digit,
  \item $sC-t\tau^{w-1}$ is a digit, and
  \item we have
    \begin{equation*}
      (A-\tau) \tau^{w-1} + (sC) = s \tau^{2w} + (-B-1)\tau^w + (sC-t\tau^{w-1}).
    \end{equation*}
  \end{enumerate}
  
\end{lemma}


\begin{figure}
  \centering
  \includegraphics{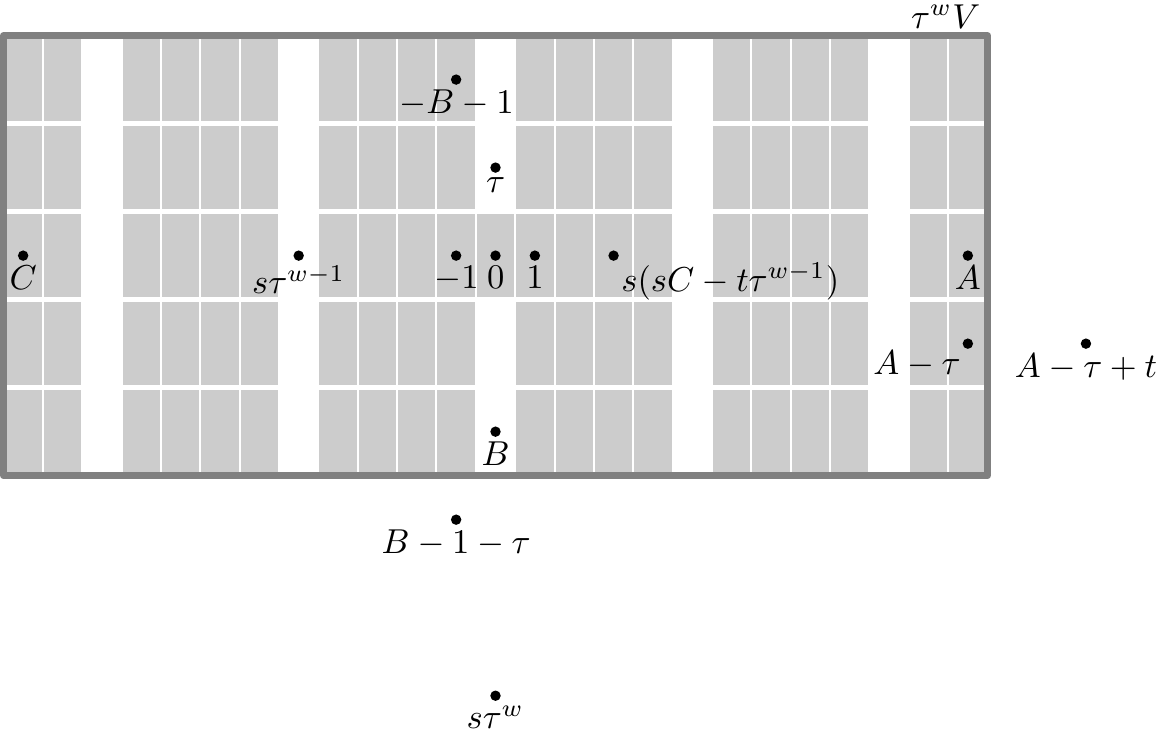}
  \caption{The $q$-is-odd situation, The figure shows the configuration $p=0$,
    $q=5$, $\tau=i\sqrt{5}$, $w=3$, $s=1$. A polygon filled grey represents a
    digit, a dot represents a points of interest in
    Lemma~\ref{lem:p-0-q-odd}.}
  \label{fig:nonopt-p0qodd}
\end{figure}


Figure~\ref{fig:nonopt-p0qodd} shows the digits used in
Lemma~\ref{lem:p-0-q-odd} for a special configuration. 


\begin{proof}
  \begin{enumerate}[(a)]

  \item See the proof of Lemma~\ref{lem:p-0-q-even}.
  \item We can rewrite the point $A$ as
    \begin{equation*}
      A = \frac12\abs\tau^{w+1} - \frac12 
      = \frac12 \left( q^{\frac12(w+1)} - 1 \right).
    \end{equation*}
    Since $q$ is odd with $q\geq3$ and $w$ is odd with $w\geq3$, we obtain
    $A\in\Z$ with $0 < A < \frac12 \abs\tau^{w+1}$ and $q \ndivides
    A$. Therefore $\tau \ndivides A$ and $A$ is in the interior of the Voronoi
    cell $\tau^wV$. The vertical (parallel to the imaginary axis)
    side-length of $\tau^wV$ is $\abs\tau^w$ and $\abs\tau < \frac12
    \abs\tau^w$, so $A-\tau$ is in the interior of $\tau^wV$, too. Since $\tau
    \ndivides A-\tau$, the element $A-\tau$ is a digit.

  \item We got $\tau \ndivides A$ and $A$ is in the interior of the Voronoi
    cell $\tau^wV$. Therefore $A$ is a digit, and---by symmetry---$sC$ is a
    digit, too.

  \item We obtain
    \begin{equation*}
      B = - \frac12 i \sqrt{q} \abs\tau^{w-1} + i \frac{1}{2} \sqrt{q}
      = \frac12 \tau \left( -\abs\tau^{w-1} + 1 \right),
    \end{equation*}
    which is inside $\tau^wV$. Therefore the same is true for $-B$. The
    horizontal (parallel to the real axis) side-length of $\tau^wV$ is larger
    than $2$, therefore $-B-1$ is inside $\tau^wV$, too. Since $\tau \divides
    B$ we get $\tau \ndivides (-B-1)$, so $-B-1$ is a digit.
  \item We obtain
    \begin{align*}
      0 < s  \left( sC - t \tau^{w-1} \right)
      &= \frac12 \left( (q+1) \abs\tau^{w-1} - \abs\tau^{w+1} + 1 \right) \\
      &= \frac12 \left( \abs\tau^{w-1} + 1 \right)
      < \frac12 \abs\tau^{w+1}.
    \end{align*}
    This means that $sC - t \tau^{w-1}$ is in the interior of the Voronoi cell
    $\tau^wV$. Since $\tau \ndivides (-A) = C$, the same is true for $sC - t
    \tau^{w-1}$, i.e.\ it is a digit.

  \item We get
    \begin{align*}
      (A-\tau) \tau^{w-1} + (sC)
      &= (A-\tau+t) \tau^{w-1} + (sC-t\tau^{w-1}) \\
      &= (B-1-\tau) \tau^w + (sC-t\tau^{w-1}) \\
      &= s \tau^{2w} + (B-1-\tau-s\tau^w) \tau^w + (sC-t\tau^{w-1}) \\
      &= s \tau^{2w} + (-B-1)\tau^w + (sC-t\tau^{w-1}),
    \end{align*}
    which can be checked easily.
    \qedhere
  \end{enumerate}
\end{proof}
 

The two lemmata above now allow us to prove the non-optimality result of this
section.


\begin{proof}[Proof of Proposition~\ref{pro:p-0-non-opt}]
  Let $q$ be even. In Lemma~\ref{lem:p-0-q-even} we got
  \begin{equation*}
    (A-1-\tau) \tau^{w-1} - s = s \tau^{2w} + (-B-1)\tau^w + (-s-\tau^{w-1})
  \end{equation*}
  and that all the coefficients there were digits, i.e.\ we have valid
  expansions on the left and right hand side. The left hand side has weight~$2$
  and is not a \wNAF{}, whereas the right hand side has weight~$3$ and is a
  \wNAF{}. Therefore a counter-example to the optimality was found.

  The case $q$ is odd works analogously. We got
  the counter-example
  \begin{equation*}
    (A-\tau) \tau^{w-1} + (sC) = s \tau^{2w} + (-B-1)\tau^w + (sC-t\tau^{w-1})
  \end{equation*}
  in Lemma~\ref{lem:p-0-q-odd}.
\end{proof}


\section{Computational Results}
\label{sec:comput-results}


\begin{figure}
  \centering
  \includegraphics[angle=90,width=\textwidth]{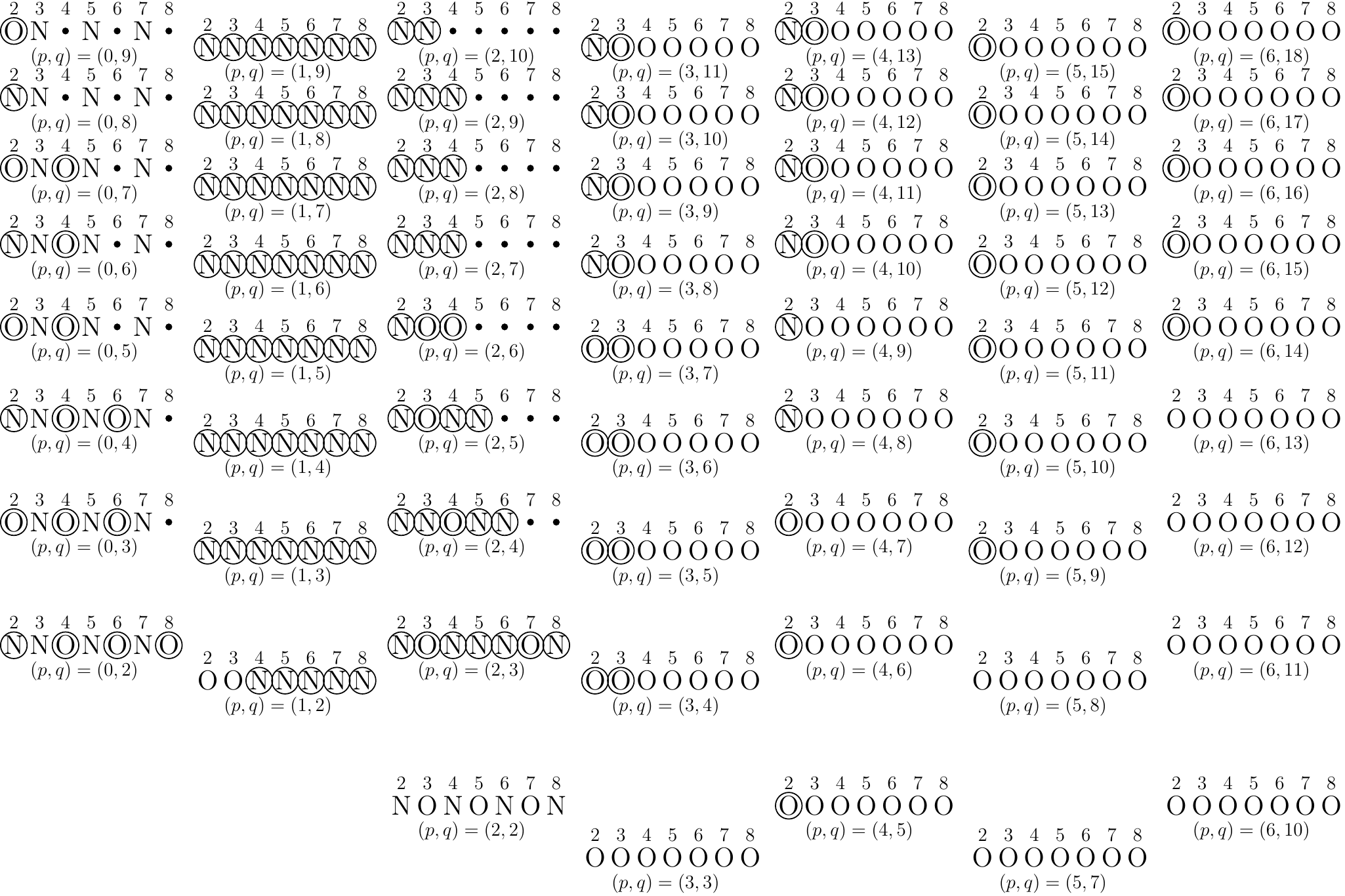}
  \caption{The optimality map including computational results. Below each block
    the parameters $p$ and $q$ (fulfilling $\tau^2-p\tau+q=0$) are printed. A
    block is positioned according to $\tau$ in the complex plane. Above each
    block are the $w$. The symbol O means that the \wNAF{}-expansions are
    optimal, N means there are non-optimal \wNAF{}-expansions. If a result is
    surrounded by a circle, then it is a computational result. Otherwise, if
    there is no circle, then the result comes from a theorem given here or was
    already known. A dot means that there is no result available.}
  \label{fig:optimality-map}
\end{figure}


This section contains computational results on the optimality of \wNAF{}s for some special
imaginary quadratic bases~$\tau$ and integers~$w$. We assume that we have a
$\tau$ coming from integers $p$ and $q$ with $q>p^2/4$. Again, we continue
looking at \wNAF{}-number systems with minimal norm representative digit set
modulo $\tau^w$ with $w\geq2$. 

As mentioned in Section~\ref{sec:abstract-optimality}, the condition
\wsubadditivity{}-condition---and therefore optimality---can be verified by
finding a \wNAF{}-expansion with weight at most~$2$ in $w \left( \card* \cD - 1
\right)$ cases. The computational results can be found in
Figure~\ref{fig:optimality-map}.


\renewcommand{\MR}[1]{}

\bibliographystyle{../shared/amsplaininitials}
\bibliography{../shared/cheub}

\providecommand{\Submitted}{Submitted} \providecommand{\availableat}{ available
  at } \providecommand{\alsoavailableat}{ also available at }
  \providecommand{\evavailableat}{ earlier version available at }
  \providecommand{\toappearin}{To appear in }
  \providecommand{\doi}[1]{\href{http://dx.doi.org/#1}{\path{doi:#1}}}
  \providecommand{\etc}{\emph{etc.}}\def\cprime{$'$}
\providecommand{\bysame}{\leavevmode\hbox to3em{\hrulefill}\thinspace}
\providecommand{\MR}{\relax\ifhmode\unskip\space\fi MR }
\providecommand{\MRhref}[2]{%
  \href{http://www.ams.org/mathscinet-getitem?mr=#1}{#2}
}
\providecommand{\href}[2]{#2}
\begin{thebibliography}{10}

\bibitem{avanzi:mywnaf}
R.~Avanzi, \emph{{A Note on the Signed Sliding Window Integer Recoding and a
  Left-to-Right Analogue}}, {Selected Areas in Cryptography: 11th International
  Workshop, SAC 2004, Waterloo, Canada, August 9-10, 2004, Revised Selected
  Papers}, Lecture Notes in Comput. Sci., vol. 3357, Springer-Verlag, Berlin,
  2004, pp.~130--143.

\bibitem{Avanzi-Heuberger-Prodinger:2006:minim-hammin}
R.~Avanzi, C.~Heuberger, and H.~Prodinger, \emph{Minimality of the {H}amming
  weight of the $\tau$-{NAF} for {K}oblitz curves and improved combination with
  point halving}, Selected Areas in Cryptography: 12th International Workshop,
  SAC 2005, Kingston, ON, Canada, August 11--12, 2005, Revised Selected Papers
  (B.~Preneel and S.~Tavares, eds.), Lecture Notes in Comput. Sci., vol. 3897,
  Springer, Berlin, 2006, pp.~332--344.

\bibitem{Avanzi-Heuberger-Prodinger:2006:scalar-multip-koblit-curves}
\bysame, \emph{Scalar multiplication on {K}oblitz curves. {U}sing the
  {F}robenius endomorphism and its combination with point halving: {E}xtensions
  and mathematical analysis}, Algorithmica \textbf{46} (2006), 249--270.

\bibitem{Avanzi-Heuberger-Prodinger:2010:arith-of}
\bysame, \emph{Arithmetic of supersingular {K}oblitz curves in characteristic
  three}, Tech. Report 2010-8, Graz University of Technology, 2010,
  \url{http://www.math.tugraz.at/fosp/pdfs/tugraz_0166.pdf}, also available as
  Cryptology ePrint Archive, Report 2010/436, \url{http://eprint.iacr.org/}.

\bibitem{Blake-Murty-Xu:ta:nonad-radix}
I.~Blake, K.~Murty, and G.~Xu, \emph{Nonadjacent radix-$\tau$ expansions of
  integers in {E}uclidean imaginary quadratic number fields}, Canad. J. Math.
  \textbf{60} (2008), no.~6, 1267--1282.

\bibitem{Blake-Kumar-Xu:2005:effic-algor}
I.~F. Blake, V.~Kumar~Murty, and G.~Xu, \emph{Efficient algorithms for
  {K}oblitz curves over fields of characteristic three}, J. Discrete Algorithms
  \textbf{3} (2005), no.~1, 113--124. \MR{MR2167767 (2006f:11069)}

\bibitem{Blake-Murty-Xu:2005:naf}
\bysame, \emph{A note on window $\tau$-{NAF} algorithm}, Inform. Process. Lett.
  \textbf{95} (2005), 496--502.

\bibitem{German-Kovacs:2007:number-system-const}
L.~Germ{\'a}n and A.~Kov{\'a}cs, \emph{On number system constructions}, Acta
  Math. Hungar. \textbf{115} (2007), no.~1-2, 155--167. \MR{MR2316627
  (2008a:11010)}

\bibitem{Gordon:1998}
D.~M. Gordon, \emph{A survey of fast exponentiation methods}, J. Algorithms
  \textbf{27} (1998), 129--146. \MR{99g:94014}

\bibitem{Heuberger:2010:nonoptimality}
C.~Heuberger, \emph{Redundant $\tau$-adic expansions {II}: {N}on-optimality and
  chaotic behaviour}, Math. Comput. Sci. \textbf{3} (2010), 141--157.

\bibitem{Heuberger-Krenn:2010:wnaf-analysis}
C.~Heuberger and D.~Krenn, \emph{Analysis of width-$w$ non-adjacent forms to
  imaginary quadratic bases}, Tech. Report 2010-10, Graz University of
  Technology, 2010, \availableat
  \url{http://www.math.tugraz.at/fosp/pdfs/tugraz_0168.pdf},\alsoavailableat
  arXiv:1009.0488v2 [math.NT].

\bibitem{Heuberger-Prodinger:2006:analy-alter}
C.~Heuberger and H.~Prodinger, \emph{Analysis of alternative digit sets for
  nonadjacent representations}, Monatsh. Math. \textbf{147} (2006), 219--248.

\bibitem{Jedwab-Mitchell:1989}
J.~Jedwab and C.~J. Mitchell, \emph{{Minimum weight modified signed-digit
  representations and fast exponentiation}}, Electron. Lett. \textbf{25}
  (1989), 1171--1172.

\bibitem{Knuth:1998:Art:2}
D.~E. Knuth, \emph{Seminumerical algorithms}, third ed., The Art of Computer
  Programming, vol.~2, Addison-Wesley, 1998.

\bibitem{Koblitz:1992:cm}
N.~Koblitz, \emph{C{M}-curves with good cryptographic properties}, Advances in
  cryptology---CRYPTO '91 (Santa Barbara, CA, 1991), Lecture Notes in Comput.
  Sci., vol. 576, Springer, Berlin, 1992, pp.~279--287. \MR{94e:11134}

\bibitem{Koblitz:1998:ellip-curve}
\bysame, \emph{An elliptic curve implementation of the finite field digital
  signature algorithm}, Advances in cryptology---CRYPTO '98 (Santa Barbara, CA,
  1998), Lecture Notes in Comput. Sci., vol. 1462, Springer, Berlin, 1998,
  pp.~327--337. \MR{MR1670960 (99j:94052)}

\bibitem{Kroell:ta:optim-of}
M.~Kr{\"o}ll, \emph{Optimality of digital expansions to the base of the
  {F}robenius endomorphism on {K}oblitz curves in characteristic three}, Tech.
  Report 2010-09, Graz University of Technology, 2010, \availableat
  \url{http://www.math.tugraz.at/fosp/pdfs/tugraz_0167.pdf}.

\bibitem{Meier-Staffelbach:1993:effic}
W.~Meier and O.~Staffelbach, \emph{Efficient multiplication on certain
  nonsupersingular elliptic curves}, Advances in cryptology---CRYPTO '92 (Santa
  Barbara, CA, 1992), Lecture Notes in Comput. Sci., vol. 740, Springer,
  Berlin, 1993, pp.~333--344. \MR{MR1287863 (95e:94037)}

\bibitem{Muir-Stinson:2005:new-minim}
J.~A. Muir and D.~R. Stinson, \emph{New minimal weight representations for
  left-to-right window methods}, Topics in Cryptology --- CT-RSA 2005 The
  Cryptographers' Track at the RSA Conference 2005, San Francisco, CA, USA,
  February 14--18, 2005, Proceedings (A.~J. Menezes, ed.), Lecture Notes in
  Comput. Sci., vol. 3376, Springer, Berlin, 2005, pp.~366--384.

\bibitem{muirstinson:minimality}
\bysame, \emph{Minimality and other properties of the width-{$w$} nonadjacent
  form}, Math. Comp. \textbf{75} (2006), 369--384. \MR{MR2176404}

\bibitem{Phillips-Burgess:2004:minim-weigh}
B.~Phillips and N.~Burgess, \emph{Minimal weight digit set conversions}, IEEE
  Trans. Comput. \textbf{53} (2004), 666--677.

\bibitem{Reitwiesner:1960}
G.~W. Reitwiesner, \emph{Binary arithmetic}, Advances in computers, vol.~1,
  Academic Press, New York, 1960, pp.~231--308.

\bibitem{Solinas:1997:improved-algorithm}
J.~A. Solinas, \emph{An improved algorithm for arithmetic on a family of
  elliptic curves}, Advances in Cryptology --- {CRYPTO} '97. 17th annual
  international cryptology conference. Santa Barbara, {CA}, {USA}. August
  17--21, 1997. Proceedings (B.~S. Kaliski, jun., ed.), Lecture Notes in
  Comput. Sci., vol. 1294, Springer, Berlin, 1997, pp.~357--371.

\bibitem{Solinas:2000:effic-koblit}
\bysame, \emph{Efficient arithmetic on {K}oblitz curves}, Des. Codes Cryptogr.
  \textbf{19} (2000), 195--249. \MR{2002k:14039}

\bibitem{Woestijne:2009:struc-of}
C.~van~de Woestijne, \emph{The structure of {A}belian groups supporting a
  number system (extended abstract)}, Actes des rencontres du CIRM \textbf{1}
  (2009), no.~1, 75--79.

\end{thebibliography}


\end{document}